\documentclass{amsart}

\usepackage{amsmath,amssymb,amsthm,mathtools}
\usepackage{enumitem}
\usepackage{array,longtable}
\usepackage{ifpdf}
\ifpdf
\else

\fi
\usepackage{tikz-cd}
\ifpdf
\usepackage[hidelinks]{hyperref}
\else
\usepackage[dvipdfmx,hidelinks]{hyperref}
\fi
\usepackage{cleveref}

\theoremstyle{plain}
\newtheorem{theorem}{Theorem}[section]
\newtheorem{proposition}[theorem]{Proposition}
\newtheorem{lemma}[theorem]{Lemma}
\newtheorem{corollary}[theorem]{Corollary}
\crefname{theorem}{Theorem}{Theorems}
\crefname{proposition}{Proposition}{Propositions}
\crefname{lemma}{Lemma}{Lemmas}
\crefname{corollary}{Corollary}{Corollaries}

\theoremstyle{definition}
\newtheorem{definition}[theorem]{Definition}
\newtheorem*{definition*}{Definition}
\crefname{definition}{Definition}{Definitions}

\theoremstyle{remark}
\newtheorem{remark}[theorem]{Remark}
\crefname{remark}{Remark}{Remarks}
\crefname{section}{Section}{Sections}
\crefname{equation}{Equation}{Equations}
\crefname{figure}{Figure}{Figures}

\newcommand{\R}{\mathbb{R}}
\newcommand{\D}{\mathcal{D}}
\newcommand{\X}{\mathcal{X}}
\newcommand{\Xp}{\mathcal{X}^{+}}
\newcommand{\M}{\mathcal{M}}
\newcommand{\DM}{\mathcal{DM}}
\renewcommand{\L}{\mathcal{L}}
\newcommand{\TB}{\mathcal{TB}}
\newcommand{\Lup}{\operatorname{Lip}^{+}_{1}(\R)}
\newcommand{\lipplus}{\operatorname{Lip}^{+}_{1}}
\DeclareMathOperator{\lipone}{Lip_1}
\DeclareMathOperator{\dis}{dis}
\DeclareMathOperator{\supp}{supp}
\DeclareMathOperator{\pr}{pr}
\newcommand{\dconc}{d_{\operatorname{conc}}}
\newcommand{\dpr}{d_{\operatorname{P}}}
\newcommand{\kf}{d_{\operatorname{KF}}}
\newcommand{\dinf}[1]{d^{#1}_{\infty}}
\newcommand{\haus}[1]{\left(#1\right)_H}
\newcommand{\Tplan}{\mathcal{T}}

\title[Pyramidal Compactification via Adjoint Transport]{Pyramidal Compactification of Asymmetric Metric Measure Spaces via Adjoint Transport}
\author{Shigeaki Yokota}
\date{}

\begin{document}

\begin{abstract}
A quasi-metric measure space (qm-space) is a set with a directed distance whose symmetrization is a complete separable metric, together with a Borel probability measure of full support. The space of pyramids of qm-spaces is a compact metric space. The associated-pyramid map from qm-spaces with the concentration metric to pyramids with a metric that metrizes weak convergence is a topological embedding with Lipschitz constant at most one and dense image. The space of qm-spaces with the box distance is complete and separable.
\end{abstract}

\maketitle

\section{Introduction}
\label{sec:introduction}
Gromov introduced \emph{mm-spaces} \cite[Definition~3$\frac12$.1, p.~113]{gromov2007met} and developed their concentration theory \cite[Chapter~3$\frac12$]{gromov2007met}. The concentration topology is strictly weaker than the measured Gromov--Hausdorff topology. For a fibration whose fibers form a L\'evy family, concentration collapses the fibers to points while retaining the base as the limit \cite[Introduction]{shioya2016mmg}.

Gromov introduced the observable distance for concentration convergence. He also introduced the box distance and the \emph{pyramid compactification} \cite[Definitions~3$\frac12$.2--3 and~3$\frac12$.45; Sections~3$\frac12$.51--53, pp.~212--214]{gromov2007met}. The classical metric that metrizes weak convergence of pyramids makes their space compact \cite[Definition~4.5 and Theorem~4.6]{shioya2022sugaku}.

A \emph{qm-space} is a set equipped with a directed distance whose symmetrization is a complete separable metric and with a Borel probability measure of full support. We write $\Xp$ for the set of its isomorphism classes. Symmetrization retains only the larger of the two directed distances between a pair of points and therefore forgets which endpoint carries an increment.

Consider the set $X=\{a,b,c\}$ with masses $0.6,0.2,0.2$. Define
\[
d_X(b,a)=d_X(c,a)=0.1,
\]
and let every other nonzero directed distance be $2$. This finite obstruction is shown in \Cref{fig:three-point-directed-obstruction}.

\begin{figure}[htbp]
\centering
\begin{tikzpicture}[>=stealth]
\node (a) at (0,1.5) {$a\ (0.6)$};
\node (b) at (-1.5,0) {$b\ (0.2)$};
\node (c) at (1.5,0) {$c\ (0.2)$};
\draw[->] (b) -- node[left] {$0.1$} (a);
\draw[->] (c) -- node[right] {$0.1$} (a);
\end{tikzpicture}
\caption{The three-point qm-space $X$. The displayed arrows give $d_X(b,a)=d_X(c,a)=0.1$. Every other nonzero directed distance is $2$.}
\label{fig:three-point-directed-obstruction}
\end{figure}
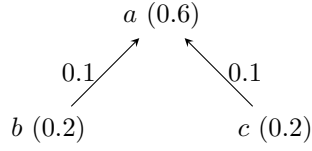

After $b$ and $c$ are identified, the lower set on the symmetric side contains the two-point mm-space of masses $0.6,0.4$ and distance $1$. The corresponding qm-space does not lie in the directed lower set because its reverse distance exceeds $0.1$. Hence merely pulling the classical pyramid compactification back along symmetrization cannot recover domination of qm-spaces.

We write $Y\preceq X$ when there is a measure-preserving $1$-Lipschitz map from a qm-space $X$ to a qm-space $Y$, and set
\[\mathcal P(X)\coloneqq\{Y\in\Xp\mid Y\preceq X\}.\]
A nonempty subset of $\Xp$ is a pyramid if it is downward closed under $\preceq$, closed in the box distance, and contains a common upper bound for every pair of its elements. Let $\Pi^+$ be the set of all such pyramids, and let $\dconc$ denote the concentration distance defined by one-sided observables.

\begin{theorem}
\label{thm:intro-qm-pyramidal-compactification}
For every $X\in\Xp$, the lower set $\mathcal P(X)$ belongs to $\Pi^+$. There exists a metric $d_{\Pi}^{\M}$ on $\Pi^+$ with the following properties.
\begin{enumerate}[label=\textup{(\roman*)}]
\item The metric space $(\Pi^+,d_{\Pi}^{\M})$ is compact.
\item For all $X,Y\in\Xp$,
\[d_{\Pi}^{\M}(\mathcal P(X),\mathcal P(Y))\leq\dconc(X,Y),\]
and $X\mapsto\mathcal P(X)$ is a topological embedding of $(\Xp,\dconc)$ into $(\Pi^+,d_{\Pi}^{\M})$.
\item The image $\{\mathcal P(X)\mid X\in\Xp\}$ is dense in $(\Pi^+,d_{\Pi}^{\M})$.
\end{enumerate}
\end{theorem}

In addition, the space of all qm-spaces is complete and separable in the box distance.

The first proof mechanism is \emph{distance recovery by one-sided observables}. Symmetrization loses endpoint order. The ordered increments of one-sided $1$-Lipschitz functions recover the directed distance completely \cite{romaguera2000semi}. This recovery yields a representation--reconstruction adjunction from qm-spaces to \emph{gd-sets} whose representation functor is fully faithful \cite[Definition~3.1]{gds1}.

The second mechanism is \emph{transport by a pyramidal adjunction}. A unit that is an isomorphism, a box-isometric left adjoint, a box-nonexpansive right adjoint, and domination refinement transport weak convergence of pyramids from the gd-set side to the qm-space side and reflect it back. \emph{Ordered finite measurements} then metrize this weak convergence and yield compactness and the Lipschitz estimate.

\Cref{sec:preliminaries} fixes mm-spaces, qm-spaces, gd-sets, the box distance, and the inputs from the upstream papers. \Cref{sec:directed-recovery} recovers the directed distance from one-sided observables and proves the representation--reconstruction adjunction and its box estimates. \Cref{sec:abstract-transport} introduces pyramidal adjunctions and proves the abstract transport theorem for pyramids and weak convergence. \Cref{sec:weak-compactness} establishes downward closedness and directedness of limit sets and derives weak compactness. \Cref{sec:ordered-measurements} metrizes weak convergence by ordered finite measurements. These results are combined in \Cref{sec:compactification} to prove \Cref{thm:intro-qm-pyramidal-compactification}. Finally, \Cref{sec:symmetric-applications} explains compatibility with the symmetric theory and how the Funk-model analysis depends on the present theory. The appendices contain the compactness and box estimates and the general theory and technical proofs for ordered measurements.

\section{Preliminaries}
\label{sec:preliminaries}

This section fixes the notation and the definitions of mm-spaces, qm-spaces, gd-sets, and the box distance, together with the upstream inputs used in distance recovery, adjoint transport, weak compactness, and finite measurements.

\subsection{Notation and standing conventions}

We write $\R$ for the real line, $\supp\mu$ for the support of a measure $\mu$, and $u_*\mu$ for the pushforward of $\mu$ by a measurable map $u$. Once a category of spaces or gd-sets has been introduced, the same symbol is also used for its set of isomorphism classes, and we do not distinguish notationally between an object and its isomorphism class. We follow \cite[Definitions~1.14, 1.20, and~1.23]{shioya2016mmg} for the Ky Fan and Prokhorov distances and for couplings.

For a distance or pseudodistance $d$ and a nonempty subset $A$, write
\[d(x,A)\coloneqq\inf_{a\in A}d(x,a),\]
and denote the Hausdorff distance or pseudodistance induced by $d$ on sets by $\haus{d}$. For real-valued Borel functions $h,k$ on a probability space $(W,\omega)$, set
\[\kf^\omega(h,k)\coloneqq\inf\{\varepsilon>0\mid\omega(|h-k|>\varepsilon)\leq\varepsilon\}.\]
This is the Ky Fan metric with respect to $\omega$.
For a measured space $X$ with distinguished probability measure $\mu_X$, we abbreviate $\kf^{\mu_X}$ to $\kf^X$.
We write $\dpr$ for the Prokhorov distance whenever the underlying metric is clear.

We write $\Tplan(\mu,\nu)$ for the set of all couplings of Borel probability measures $\mu$ and $\nu$, and $\pr_1,\pr_2$ for the projections from a product onto its factors. For functions pulled back to a coupling $\pi$, we use $\kf^\pi$.

\subsection{mm-spaces, qm-spaces, and gd-sets}
For an mm-space $X$, let $\lipone(X)$ denote the set of real-valued $1$-Lipschitz functions on $X$.

\begin{definition}[mm-space {\cite[Definition~3$\frac12$.1, p.~113]{gromov2007met}}]
\label{def:mm-space}
A triple $(X,d_X,\mu_X)$ is called an \emph{mm-space}, or simply $X$ is called an mm-space, if $(X,d_X)$ is a complete separable metric space and $\mu_X$ is a Borel probability measure on $(X,d_X)$ satisfying $\supp\mu_X=X$.
\end{definition}

\begin{definition*}[Isomorphism of mm-spaces {\cite[pp.~117--118]{gromov2007met}\cite[Definition~2.10]{shioya2016mmg}}]
Two mm-spaces $X$ and $Y$ are \emph{isomorphic} if there exists a measure-preserving isometric bijection between them. We denote the category of mm-spaces and measure-preserving $1$-Lipschitz maps by $\X$.
\end{definition*}

\begin{definition*}[Domination of mm-spaces {\cite[Definition~3$\frac12$.15, p.~134]{gromov2007met} \cite[Definition~2.10]{shioya2016mmg}}]
We write $Y\preceq X$ if there exists a measure-preserving $1$-Lipschitz map $u\colon X\to Y$, and in this case we say that $X$ \emph{dominates} $Y$.
\end{definition*}

\begin{definition}[qm-space]
\label{def:asymmetric-mm-space}
A triple $(X,d_X,\mu_X)$ is called a \emph{qm-space}, or simply $X$ is called a qm-space, if the following three conditions hold.
\begin{enumerate}[label=\textup{(\roman*)}]
\item The map $d_X\colon X\times X\to[0,+\infty)$ satisfies $d_X(x,x)=0$ and the triangle inequality.
\item The symmetrization $d_X^{\mathrm s}(x,y)\coloneqq\max\{d_X(x,y),d_X(y,x)\}$ is a complete separable metric on $X$.
\item The measure $\mu_X$ is a Borel probability measure on $(X,d_X^{\mathrm s})$ with $\supp\mu_X=X$.
\end{enumerate}
\end{definition}
The abbreviation qm-space stands for quasi-metric measure space.
For related measured quasi-metric spaces, the relation between forward and backward concentration functions and the tails of one-sided Lipschitz functions, together with an application to asymmetric biological-sequence similarity, is studied in \cite[Definition~3.2 and Lemma~3.6]{stojmirovic2004quasi}. Gromov--Hausdorff-type convergence in irreversible metric-measure geometry is treated in \cite{kristaly-zhao2022geometry}. The present paper instead concerns concentration topology and pyramids.
The triangle inequality gives
\[|d_X(x,y)-d_X(x',y')|\leq d_X^{\mathrm s}(x,x')+d_X^{\mathrm s}(y,y'),\]
and therefore $d_X$ is uniformly continuous with respect to the product metric induced by $d_X^{\mathrm s}$.

\begin{definition}[Isomorphism of qm-spaces]
\label{def:asymmetric-mm-space-isomorphism}
Two qm-spaces $X$ and $Y$ are \emph{isomorphic} if there exists a bijection $\varphi\colon X\to Y$ such that
\[d_Y(\varphi(x),\varphi(x'))=d_X(x,x')\quad(x,x'\in X),\qquad \varphi_*\mu_X=\mu_Y.\]
With measure-preserving $1$-Lipschitz maps as morphisms, $\Xp$ is also regarded as a category.
\end{definition}

The $1$-Lipschitz inequalities defining morphisms of mm-spaces and qm-spaces can be expressed in the same language by using families of real-valued functions. This will allow us to describe the morphism condition by pullback containment of function families.
Restricted-observable constructions and representations of measure spaces by function families appear in \cite{pestov2008gds,hanika2022gds}. The gd-set formalism provides the form of this representation used here.

\begin{definition}[Geometric data set {\cite[Definition~3.1]{gds1}}]
\label{def:gds}
A triple $(X,F_X,\mu_X)$ is called a \emph{geometric data set}, or simply $X$ is called a geometric data set, if $F_X$ is a nonempty family of real-valued functions on $X$,
\[d_{F_X}(x,x')\coloneqq\sup_{f\in F_X}|f(x)-f(x')|\]
is a complete separable metric on $X$, and $\mu_X$ is a Borel probability measure on $(X,d_{F_X})$ satisfying $\supp\mu_X=X$.
\end{definition}
In this paper, a geometric data set without any additional closure condition is simply called a gd-set.

We denote by $\overline{F_X}$ the closure of $F_X$ with respect to pointwise convergence.
\begin{definition*}[Isomorphism of gd-sets {\cite[Lemma~2.2, Definitions~3.1 and~3.3, and Proposition~3.10]{gds1}}]
Two gd-sets $X$ and $Y$ are isomorphic if there exists a Borel measurable map $u\colon X\to Y$ such that $u_*\mu_X=\mu_Y$ and $\overline{F_Y}\circ u=\overline{F_X}$. This condition implies that $u$ is a measure-preserving surjective isometry with respect to $d_{F_X}$ and $d_{F_Y}$. We denote the set of isomorphism classes of gd-sets by $\D$.
\end{definition*}
The surjectivity assertion uses the support formula
\[
\supp f_*\mu=\overline{f(\supp\mu)}
\]
from \cite[Lemma~2.2]{gds1} for a Borel measure on a second-countable space and a continuous map $f$.

\begin{definition*}[Domination of gd-sets {\cite[Definition~3.8]{gds1}}]
If a Borel measurable map $u\colon X\to Y$ satisfies
\[u_*\mu_X=\mu_Y,\qquad F_Y\circ u\subset\overline{F_X},\]
then we write $Y\preceq X$, say that $X$ \emph{dominates} $Y$, and call $u$ a \emph{domination}.
\end{definition*}

Postcomposition closure of the feature families is needed in the compactness theory. We therefore introduce the following classes of scalar maps.

\begin{definition*}[Monoidal subfamilies {\cite[Definition~3.1]{gds2}}]
A subfamily $\L\subset\lipone(\R)$ is \emph{monoidal} if it contains the identity map and is closed under composition and pointwise convergence.
\end{definition*}
For a real-valued function $f$ and a family $F$ of such functions, write
\[
\L\circ f\coloneqq\{p\circ f\mid p\in\L\},\qquad
\L\circ F\coloneqq\{p\circ f\mid p\in\L,\ f\in F\}.
\]
We write the family of translations with clipping as
\[
\TB\coloneqq\{t\mapsto\max\{l,\min\{t+c,u\}\}\mid c\in\R,\ l\in[-\infty,+\infty),\ u\in(-\infty,+\infty],\ l\leq u\}
\]
\cite[Definitions~3.1 and~3.2]{gds2}. Every monoidal subfamily $\L$ considered in this paper is assumed to satisfy $\TB\subset\L$.

\begin{definition*}[$\L$-gd-sets {\cite[Definition~3.5]{gds2}}]
A gd-set $X$ is called an \emph{$\L$-gd-set} if
\[\L\circ\overline{F_X}\subset\overline{F_X}.\]
\end{definition*}

For a gd-set $X$, write
\[\L\circ X\coloneqq(X,\L\circ F_X,\mu_X)\]
and call it the \emph{$\L$-saturation} of $X$. Since the identity map belongs to $\L$ and every member of $\L$ is $1$-Lipschitz, the induced metric of $\L\circ X$ equals $d_{F_X}$. Monoidality implies that $\L\circ X$ is an $\L$-gd-set. Under the standing assumption $\TB\subset\L$, GDS II identifies the isomorphism classes of $\L$-gd-sets with the class denoted by $\D/\L$ \cite[Theorem~3.19]{gds2}. We denote this class by $\L\circ\D$. We use, in particular, $\lipone(\R)$ and
\[\Lup=\{p\colon\R\to\R\mid p\text{ is nondecreasing and $1$-Lipschitz}\}.\]
Both are monoidal subfamilies containing $\TB$.

We regard $\L\circ\D$ as the category of $\L$-gd-sets and dominations of gd-sets.

\subsection{Box distance on gd-sets}

For gd-sets $X,Y$, a closed set $S\subset X\times Y$, and real-valued functions $h,k$ on $X\times Y$, write
\[\dinf{S}(h,k)\coloneqq\sup_{(x,y)\in S}|h(x,y)-k(x,y)|.\]
This is a pseudometric on the family of all functions on $X\times Y$, and we use $\haus{\dinf{S}}$ on function families. When $S=\emptyset$, this Hausdorff pseudometric is set equal to $0$. For a subset $A\subset M$ of a metric space $(M,d)$, write its closed $\varepsilon$-neighborhood as
\[\mathrm B(A,\varepsilon;d)\coloneqq\{x\in M\mid d(x,A)\leq\varepsilon\}.\]

\begin{definition}[Observable distance on gd-sets {\cite[Theorem~4.6]{gds1}}]
\label{def:gds-distances}
The \emph{observable distance} between gd-sets $X$ and $Y$ is defined by
\begin{equation}
\label{eq:gds-dconc}
\dconc(X,Y)\coloneqq\inf_{\pi\in\Tplan(\mu_X,\mu_Y)}\haus{\kf^\pi}(F_X\circ\pr_1,F_Y\circ\pr_2).
\end{equation}
\end{definition}
By \cite[Theorem~4.6]{gds1}, the infimum in \Cref{eq:gds-dconc} is attained.

\begin{definition}[Box distance on gd-sets {\cite[Definition~5.7, Lemma~5.12, and Theorem~5.14]{gds1}}]
\label{def:gds-box-distance}
The \emph{box distance} is defined by
\begin{equation}
\label{eq:gds-box}
\begin{aligned}
\Box(X,Y)&\coloneqq\inf\left\{\max\{1-\pi(S),2\haus{\dinf{S}}(\overline{F_X}\circ\pr_1,\overline{F_Y}\circ\pr_2)\}\ \mid\right.\\
&\left.\begin{array}{l}
\pi\in\Tplan(\mu_X,\mu_Y),\\
S\subset X\times Y\text{ is closed}
\end{array}\right\}.
\end{aligned}
\end{equation}
\end{definition}
By \cite[Lemma~5.12]{gds1}, replacing the feature families in this formula by their pointwise closures does not change its value. By \cite[Theorem~5.14]{gds1}, the infimum in \Cref{eq:gds-box} is attained. Moreover, \cite[Proposition~5.8]{gds1} gives
\[
\dconc(X,Y)\leq\Box(X,Y)
\]
for all gd-sets $X$ and $Y$.

\subsection{Fixed inputs from the upstream papers}

We abbreviate \cite{gds1} to GDS I and \cite{gds2} to GDS II. \Cref{tab:upstream-inputs} groups the statements used from these papers by their first use in the argument.

\begin{table}[htbp]
\centering
\small
\setlength{\tabcolsep}{3pt}
\begin{tabular}{>{\raggedright\arraybackslash}p{.22\textwidth}>{\raggedright\arraybackslash}p{.16\textwidth}>{\raggedright\arraybackslash}p{.29\textwidth}>{\raggedright\arraybackslash}p{.23\textwidth}}
\hline
Input & Source & Purpose & Locator\\
\hline
gd-sets, isomorphism, and domination & GDS I & Representation--reconstruction adjunction & Definition~3.1; Lemma~2.2; Definitions~3.3, 3.8\\
Concentration and box distances & GDS I & Pullback metrics and closedness & Theorem~4.6; Definitions~5.2, 5.7; Proposition~5.8; Lemma~5.12; Proposition~5.13; Theorem~5.14\\
Antisymmetry of feature order & GDS I & Separation for pullback metrics & Proposition~3.9\\
Metric control of dominations & GDS I & Metric consequences of gd-set isomorphism & Proposition~3.10\\
Closedness of domination under concentration & GDS I & Separation for the pullback concentration metric & Theorem~4.16\\
Compactness of fixed-marginal couplings & GDS I & Limit coupling in box-closedness & Lemma~2.10\\
Monoidal subfamilies and completeness and separability & GDS II & Category used for transport & Definitions~3.1--3.3, 3.5, 3.8; Theorems~3.19, 4.10\\
Domination and common-upper-bound refinement & GDS II & Downward closedness and directedness of weak limits & Lemmas~5.4, 5.6\\
Pyramids on metrized categories & GDS II & Abstract pyramid transport & Definition~5.1\\
Weak Hausdorff subsequence theorem & GDS II & Sequential compactness of weak limits & Lemma~2.11\\
Finite measurements & GDS II & Comparison with ordered measurements & Lemma~7.1; Propositions~7.2--7.3\\
\hline
\end{tabular}
\caption{Upstream statements and their roles in the present paper.}
\label{tab:upstream-inputs}
\end{table}

\section{One-sided observables and directed-distance recovery}
\label{sec:directed-recovery}

Symmetrization forgets which endpoint carries an increment. One-sided observables retain this order and recover the directed distance. Their representation functor is fully faithful in the representation--reconstruction adjunction proved below.

For a qm-space $X$, the relevant observable family consists of functions controlled by ordered increments. This is the semi-Lipschitz function class studied by Romaguera and Sanchis \cite{romaguera2000semi} in the setting of quasi-metric spaces.

\begin{definition}[One-sided $1$-Lipschitz functions]
\label{def:lipplus}
For a qm-space $X$, define
\[\lipplus(X)\coloneqq\{f\colon X\to\R\mid f(y)-f(x)\leq d_X(x,y)\text{ for all $x,y\in X$}\}.\]
\end{definition}

Let $I(-)\colon\X\to\Xp$ denote the inclusion functor that regards a symmetric distance as the same directed distance. Applying the one-sided inequality for a symmetric distance in both orders gives the usual $1$-Lipschitz condition. Therefore,
\begin{equation}
\label{eq:symmetric-directed-lipschitz}
\lipplus(I(X))=\lipone(X).
\end{equation}
For the same reason, the morphisms from $I(X)$ to $I(X')$ are precisely the morphisms from $X$ to $X'$, and hence $I(-)$ is fully faithful.

For an mm-space $X$ and a qm-space $Y$, we have
\begin{align}
d_X(x,x')&=\sup_{f\in\lipone(X)}|f(x)-f(x')|,\label{eq:symmetric-recovery}\\
d_Y(y,y')&=\sup_{f\in\lipplus(Y)}\{f(y')-f(y)\}.\label{eq:directed-recovery}
\end{align}
Indeed, the right-hand side of \Cref{eq:directed-recovery} is at most the left-hand side by definition, while the reverse inequality follows by taking $f_y(z)\coloneqq d_Y(y,z)$. Furthermore,
\[\sup_{f\in\lipplus(Y)}|f(y)-f(y')|=\max\{d_Y(y,y'),d_Y(y',y)\}=d_Y^{\mathrm s}(y,y').\]
Applying \Cref{eq:directed-recovery} to $I(X)$ and using \Cref{eq:symmetric-directed-lipschitz} together with the closure of $\lipone(X)$ under sign reversal gives \Cref{eq:symmetric-recovery}.

\subsection{Representation and reconstruction of qm-spaces}
\label{sec:adjunctions}

\begin{definition}[Representation functor]
\label{def:space-to-gds}
Define
\[\operatorname{Rep}^{+}(-)\colon\Xp\longrightarrow\Lup\circ\D\]
by
\[\operatorname{Rep}^{+}(Y)\coloneqq(Y,\lipplus(Y),\mu_Y).\]
It leaves the underlying maps unchanged on morphisms.
\end{definition}

This assignment is well-defined as a functor. The metric, completeness, separability, and measure conditions required for $\operatorname{Rep}^{+}(Y)$ to be a gd-set follow from \Cref{eq:directed-recovery}. Let $p\in\Lup$ and $f\in\lipplus(Y)$. If $f(y')\geq f(y)$, then
\[p(f(y'))-p(f(y))\leq f(y')-f(y)\leq d_Y(y,y'),\]
whereas if $f(y')<f(y)$, monotonicity implies that the left-hand side is nonpositive. Therefore, $p\circ f\in\lipplus(Y)$. This function family is closed under pointwise convergence, so the resulting gd-set is an $\Lup$-gd-set.

A $1$-Lipschitz map $u\colon Y\to Y'$ pulls $\lipplus(Y')$ back into $\lipplus(Y)$. Thus, the assignment maps morphisms to morphisms and preserves identity morphisms and composition.

\begin{definition}[Reconstruction functor from gd-sets]
\label{def:gds-to-space}
For a gd-set $X\in\Lup\circ\D$, set
\[
d_X^+(x,x')\coloneqq\sup_{f\in\overline{F_X}}\{f(x')-f(x)\},\qquad
\operatorname{Rec}^{+}(X)\coloneqq(X,d_X^+,\mu_X).
\]
It leaves the underlying maps unchanged on morphisms.
\end{definition}

\begin{remark}
\label{rem:symmetric-reconstruction-bridge}
For $Z\in\lipone(\R)\circ\D$, define
\[\operatorname{Rec}(Z)\coloneqq(Z,d_{F_Z},\mu_Z).\]
The inclusion $-\operatorname{id}_{\R}\in\lipone(\R)$ implies that $f\in\overline{F_Z}$ if and only if $-f\in\overline{F_Z}$. Therefore,
\[\sup_{f\in\overline{F_Z}}\{f(z')-f(z)\}=\sup_{f\in\overline{F_Z}}|f(z')-f(z)|,\]
so the directed and symmetric reconstructions agree.
\end{remark}

\begin{theorem}
\label{thm:adjunction-direction}
There is a natural adjunction $\operatorname{Rep}^{+}\dashv\operatorname{Rec}^{+}$. More precisely, there is a natural bijection that leaves the underlying maps unchanged,
\begin{equation}
\label{eq:hom-asymmetric}
\operatorname{Hom}_{\Lup\circ\D}(\operatorname{Rep}^{+}(Y),Y')\simeq\operatorname{Hom}_{\Xp}(Y,\operatorname{Rec}^{+}(Y')).
\end{equation}
The unit of this adjunction is an isomorphism, and $\operatorname{Rep}^{+}(-)$ is fully faithful.
\end{theorem}

\begin{proof}
We first show that $\operatorname{Rec}^{+}(-)$ is a functor from $\Lup\circ\D$ to $\Xp$.

The family $\overline{F_X}$ is $\Lup$-closed, and $\Lup$ contains every constant function. Fix $f_0\in\overline{F_X}$. Every constant function belongs to $\overline{F_X}$ as the composition of a constant map with $f_0$. Therefore, $d_X^+$ is nonnegative and $d_X^+(x,x)=0$. For every $f\in\overline{F_X}$, the decomposition
\[f(x'')-f(x)=\{f(x'')-f(x')\}+\{f(x')-f(x)\}\]
and taking suprema give the triangle inequality for $d_X^+$. We also have
\[\max\{d_X^+(x,x'),d_X^+(x',x)\}=\sup_{f\in\overline{F_X}}|f(x)-f(x')|=d_{F_X}(x,x').\]
Thus, the symmetrization agrees with the metric of the gd-set, and the completeness, separability, and measure conditions also hold.

Let $u\colon X\to Y$ be a domination of gd-sets. Taking the pointwise closure of $F_Y\circ u\subset\overline{F_X}$ gives $\overline{F_Y}\circ u\subset\overline{F_X}$. Taking the supremum of the increments therefore yields $d_Y^+(u(x),u(x'))\leq d_X^+(x,x')$. Thus, $\operatorname{Rec}^{+}(-)$ maps morphisms to morphisms.
This proves functoriality.

Let $u\colon Y\to Y'$ be a measure-preserving map. Since $\lipplus(Y)$ is closed under pointwise convergence, $u$ is a morphism of gd-sets from $\operatorname{Rep}^{+}(Y)$ to $Y'$ if and only if
\[F_{Y'}\circ u\subset\lipplus(Y).\]
If this inclusion holds, then $\overline{F_{Y'}}\circ u\subset\lipplus(Y)$ because $\lipplus(Y)$ is closed under pointwise convergence. Taking the supremum of the increments, we obtain $d_{Y'}^+(u(y),u(y'))\leq d_Y(y,y')$. Conversely, if $u$ is $1$-Lipschitz, then every $f\in F_{Y'}$ satisfies
\[f(u(y'))-f(u(y))\leq d_{Y'}^+(u(y),u(y'))\leq d_Y(y,y').\]
This proves \Cref{eq:hom-asymmetric}. The bijection leaves the underlying maps unchanged and is therefore natural. The unit $Y\to\operatorname{Rec}^{+}(\operatorname{Rep}^{+}(Y))$ is the identity map and is an isomorphism by \Cref{eq:directed-recovery}. Therefore, $\operatorname{Rep}^{+}$ is fully faithful.
This completes the proof.
\end{proof}

For qm-spaces $X,Y$, specialize \Cref{eq:hom-asymmetric} to $Y=X$ and $Y'=\operatorname{Rep}^{+}(Y)$. Postcomposition with the unit isomorphism $Y\to\operatorname{Rec}^{+}(\operatorname{Rep}^{+}(Y))$ identifies the resulting Hom-set with $\operatorname{Hom}_{\Xp}(X,Y)$, giving
\[
\operatorname{Hom}_{\Xp}(X,Y)\simeq\operatorname{Hom}_{\Lup\circ\D}(\operatorname{Rep}^{+}(X),\operatorname{Rep}^{+}(Y)).
\]
Because this bijection leaves the underlying maps unchanged, an isomorphism between $\operatorname{Rep}^{+}(X)$ and $\operatorname{Rep}^{+}(Y)$ and its inverse give inverse qm-space morphisms. Thus $\operatorname{Rep}^{+}(X)\simeq\operatorname{Rep}^{+}(Y)$ implies $X\simeq Y$.

The one-sided observable representation pulls the gd-set box and concentration distances back to qm-spaces. The symmetric observable representation similarly pulls the gd-set box distance back to mm-spaces.

\begin{definition}[Box distance on qm-spaces]
\label{def:pullback-distances}
For qm-spaces $X,Y$, define
\[\Box(X,Y)\coloneqq\Box(\operatorname{Rep}^{+}(X),\operatorname{Rep}^{+}(Y)).\]
\end{definition}

\begin{definition}[Observable distance on qm-spaces]
\label{def:pullback-observable-distance}
For qm-spaces $X,Y$, define
\[\dconc(X,Y)\coloneqq\dconc(\operatorname{Rep}^{+}(X),\operatorname{Rep}^{+}(Y)).\]
\end{definition}

\begin{definition}[Representation of an mm-space]
\label{def:mm-space-representation}
For an mm-space $X$, set
\[\operatorname{Rep}(X)\coloneqq(X,\lipone(X),\mu_X).\]
\end{definition}

\begin{definition}[Box distance on mm-spaces]
\label{def:pullback-mm-box-distance}
For mm-spaces $X,Y$, define
\[\Box(X,Y)\coloneqq\Box(\operatorname{Rep}(X),\operatorname{Rep}(Y)).\]
\end{definition}

\begin{theorem}
\label{thm:pullback-metrics}
The distances $\Box$ and $\dconc$ in \Cref{def:pullback-distances,def:pullback-observable-distance} are metrics on $\Xp$, and the distance $\Box$ in \Cref{def:pullback-mm-box-distance} is a metric on $\X$.
\end{theorem}

\begin{proof}
Nonnegativity and symmetry follow from the definition. The triangle inequalities for $\Box$ and $\dconc$ on gd-sets follow from \cite[Proposition~5.13]{gds1} and \cite[Theorem~3.10]{hanika2022gds}, respectively. Since the distances in the statement are pullbacks of these gd-set distances, their triangle inequalities follow.

Let $X,Y\in\X$ and suppose that $\Box(X,Y)=0$. By \cite[Proposition~5.13]{gds1}, their symmetric observable representations are isomorphic. The distance recovery in \Cref{eq:symmetric-recovery} shows that the underlying measure-preserving bijection is an isometry, so $X\simeq Y$.

Now let $X,Y\in\Xp$ and suppose that $\rho(X,Y)=0$ for $\rho\in\{\Box,\dconc\}$. If $\rho=\Box$, then \cite[Proposition~5.13]{gds1} gives $\operatorname{Rep}^{+}(X)\simeq\operatorname{Rep}^{+}(Y)$. Suppose that $\rho=\dconc$. The closedness theorem for domination under concentration says that if gd-sets $Y_n\preceq X_n$ satisfy $\dconc(X_n,X)\to0$ and $\dconc(Y_n,Y)\to0$, then $Y\preceq X$. Set $Z\coloneqq\operatorname{Rep}^{+}(Y)$. The constant sequences $X_n=Y_n=Z$ satisfy $Y_n\preceq X_n$ by reflexivity and converge in concentration to $\operatorname{Rep}^{+}(X)$ and $\operatorname{Rep}^{+}(Y)$, respectively, because $\dconc(X,Y)=0$. The theorem gives $\operatorname{Rep}^{+}(Y)\preceq\operatorname{Rep}^{+}(X)$. Using instead the constant sequences $X_n=Y_n=\operatorname{Rep}^{+}(X)$ gives the reverse domination. The antisymmetry of $\preceq$ \cite[Proposition~3.9]{gds1} now gives $\operatorname{Rep}^{+}(X)\simeq\operatorname{Rep}^{+}(Y)$. The isomorphism in either case lifts to inverse qm-space morphisms by the two-object Hom-set identification following \Cref{thm:adjunction-direction}, so $X\simeq Y$.
This completes the proof.
\end{proof}

\subsection{Closedness of domination and box estimates for reconstruction}
\begin{proposition}
\label{prop:domination-box-closedness}
Suppose that gd-sets $X_n,Y_n,X,Y$ satisfy
\[
Y_n\preceq X_n,\qquad
\Box(X_n,X)\longrightarrow0,\qquad
\Box(Y_n,Y)\longrightarrow0.
\]
Then $Y\preceq X$.
\end{proposition}

We defer the proof to \Cref{subsec:appendix-domination-closedness}.

\subsection{Box and concentration distances}

For two qm-spaces $X,Y$ and a closed set $S\subset X\times Y$, set
\[\dis S\coloneqq\sup\{|d_X(x,x')-d_Y(y,y')|\mid(x,y),(x',y')\in S\}.\]
If $S=\emptyset$, set $\dis S\coloneqq0$ \cite[Definition~5.2]{gds1}.

The gd-set box distance controls a Hausdorff error between function families, while the reconstructed qm-space is compared through distortion of its directed metric. We first bound the space-side box distance by this distortion and then bound the distortion by the function-family error. These two estimates yield the box-nonexpansiveness of the right adjoint and the completeness of the pullback distance.

\begin{lemma}
\label{lem:asymmetric-box-upper-by-distortion}
For $A,B\in\Xp$, $\pi\in\Tplan(\mu_A,\mu_B)$, and a closed set $S\subset A\times B$, we have
\[\Box(A,B)\leq\max\{1-\pi(S),\dis S\}.\]
\end{lemma}

\begin{proof}
The assertion is immediate if $S=\emptyset$ or $\dis S=+\infty$, so assume that $S\neq\emptyset$ and $\dis S<+\infty$. For $f\in\lipplus(A)$, set
\[(E_Sf)(v)\coloneqq\inf_{(u,w)\in S}\{f(u)+d_B(w,v)\}\qquad(v\in B).\]
Fix $(x_0,y_0)\in S$. For every $(u,w)\in S$ and $v\in B$,
\[f(u)+d_B(w,v)\geq f(x_0)-\dis S-d_B(v,y_0),\]
while choosing $(u,w)=(x_0,y_0)$ in the definition gives
$(E_Sf)(v)\leq f(x_0)+d_B(y_0,v)$. Thus, $E_Sf$ takes finite real values. The triangle inequality gives
\[(E_Sf)(v')-(E_Sf)(v)\leq d_B(v,v'),\]
and therefore $E_Sf\in\lipplus(B)$. Moreover, for $(x,y)\in S$,
\[f(x)-\dis S\leq(E_Sf)(y)\leq f(x).\]
Thus, the uniform distance on $S$ between $E_Sf+(\dis S)/2$ and $f$ is at most $(\dis S)/2$. Interchanging $A$ and $B$ gives the same estimate, and hence
\[2\haus{\dinf{S}}(\lipplus(A)\circ\pr_1,\lipplus(B)\circ\pr_2)\leq\dis S.\]
The conclusion follows from \Cref{def:pullback-distances,eq:gds-box}.
This completes the proof.
\end{proof}

\begin{proposition}
\label{prop:right-adjoint-box-nonexpansive}
For $X,Y\in\Lup\circ\D$, we have
\begin{equation}
\label{eq:xp-box-nonexpansive}
\Box(\operatorname{Rec}^{+}(X),\operatorname{Rec}^{+}(Y))\leq\Box(X,Y).
\end{equation}
For $X,Y\in\lipone(\R)\circ\D$, we have
\begin{equation}
\label{eq:x-box-nonexpansive}
\Box(\operatorname{Rec}(X),\operatorname{Rec}(Y))\leq\Box(X,Y).
\end{equation}
Therefore, both $\operatorname{Rep}^{+}\circ\operatorname{Rec}^{+}$ and $\operatorname{Rep}\circ\operatorname{Rec}$ are nonexpansive with respect to the box distance between gd-sets.
\end{proposition}

\begin{proof}
Take a closed set $S\subset X\times Y$ and set
\[H_S\coloneqq\haus{\dinf{S}}(\overline{F_X}\circ\pr_1,\overline{F_Y}\circ\pr_2).\]
The estimate below is immediate if $H_S=+\infty$, so assume that $H_S<+\infty$. For $(x,y),(x',y')\in S$, $f\in\overline{F_X}$, and any $\varepsilon>0$, take $g\in\overline{F_Y}$ such that
\[\dinf{S}(f\circ\pr_1,g\circ\pr_2)<H_S+\varepsilon.\]
Then
\begin{align*}
f(x')-f(x)&\leq g(y')-g(y)+2H_S+2\varepsilon\\
&\leq d_Y^+(y,y')+2H_S+2\varepsilon.
\end{align*}
Taking the supremum, letting $\varepsilon\downarrow0$, and then interchanging $X$ and $Y$, we obtain
\[\dis S\leq2H_S,\]
where $\dis S$ is the distortion with respect to the directed distances of the two reconstructed spaces. Applying \Cref{lem:asymmetric-box-upper-by-distortion} for any $\pi\in\Tplan(\mu_X,\mu_Y)$ gives
\[
\Box(\operatorname{Rec}^{+}(X),\operatorname{Rec}^{+}(Y))
\leq\max\{1-\pi(S),2H_S\}.
\]
Taking the infimum over $\pi$ and $S$ proves \Cref{eq:xp-box-nonexpansive}.

Now let $X,Y\in\lipone(\R)\circ\D$, and use the $H_S$ defined above for a closed set $S\subset X\times Y$. The estimate below is immediate if $S=\emptyset$ or $H_S=+\infty$, so assume that $S\neq\emptyset$ and $H_S<+\infty$. For $(x,y),(x',y')\in S$, $f\in\overline{F_X}$, and any $\varepsilon>0$, take $g\in\overline{F_Y}$ such that
\[\dinf{S}(f\circ\pr_1,g\circ\pr_2)<H_S+\varepsilon.\]
Then
\[
|f(x)-f(x')|\leq |g(y)-g(y')|+2H_S+2\varepsilon
\leq d_{F_Y}(y,y')+2H_S+2\varepsilon.
\]
Taking the supremum, letting $\varepsilon\downarrow0$, and then interchanging $X$ and $Y$, we obtain
\[
\dis_{\mathrm s}S\coloneqq
\sup_{\substack{(x,y),(x',y')\in S}}
\bigl|d_{F_X}(x,x')-d_{F_Y}(y,y')\bigr|
\leq2H_S.
\]
Regard the reconstructed mm-spaces as qm-spaces through $I$. By \Cref{def:pullback-distances,def:pullback-mm-box-distance,eq:symmetric-directed-lipschitz}, their box distance is unchanged. Applying \Cref{lem:asymmetric-box-upper-by-distortion} for any $\pi\in\Tplan(\mu_X,\mu_Y)$ gives
\[
\Box(\operatorname{Rec}(X),\operatorname{Rec}(Y))
\leq\max\{1-\pi(S),\dis_{\mathrm s}S\}
\leq\max\{1-\pi(S),2H_S\}.
\]
Taking the infimum over $\pi$ and $S$ and using \Cref{eq:gds-box} proves \Cref{eq:x-box-nonexpansive}. The assertions about the composites follow from \Cref{def:pullback-distances,def:pullback-mm-box-distance}.
This completes the proof.
\end{proof}

\begin{corollary}
\label{cor:pullback-box-polish}
The metric space $(\Xp,\Box)$ is complete and separable.
\end{corollary}

\begin{proof}
The space $\Lup\circ\D$ is complete and separable \cite[Theorem~4.10]{gds2}. By \Cref{prop:right-adjoint-box-nonexpansive}, $(\operatorname{Rep}^{+}\circ\operatorname{Rec}^{+})(-)$ is nonexpansive, and it is idempotent because the unit is an isomorphism. Therefore, its image is a closed separable subspace of $\Lup\circ\D$. By \Cref{def:pullback-distances}, this image is isometric to $\Xp$, which proves the result.
This completes the proof.
\end{proof}

\section{Pyramidal adjunctions and transport of weak convergence}
\label{sec:abstract-transport}

Weak convergence is initially available only on the gd-set side. We establish conditions under which an adjunction transports it back to qm-spaces.

\begin{definition}[$\Box$-metrized category]
\label{def:metrized-category}
A pair $(\mathcal C,\Box_{\mathcal C})$ is called a \emph{$\Box$-metrized category} if $\mathcal C$ is a category and
\[\Box_{\mathcal C}\colon\mathcal C\times\mathcal C\longrightarrow[0,+\infty)\]
is a distance on the isomorphism classes of its objects. We write $A\preceq B$ if there exists a morphism from $B$ to $A$.
\end{definition}
When the category is clear, we write $\Box$ for $\Box_{\mathcal C}$.

\begin{definition}[Pyramid on a category {\cite[Definition~5.1]{gds2}}]
\label{def:categorical-pyramid}
Let $\mathcal C$ be a $\Box$-metrized category. A full subcategory $\mathcal P$ of $\mathcal C$ is called a \emph{pyramid} if it satisfies the following conditions.
\begin{enumerate}[label=\textup{(\roman*)}]
\item If $A\preceq B\in\mathcal P$, then $A\in\mathcal P$.
\item If $A,B\in\mathcal P$, then there exists $C\in\mathcal P$ such that $A\preceq C$ and $B\preceq C$.
\item The object set of $\mathcal P$ is nonempty and closed with respect to $\Box_{\mathcal C}$.
\end{enumerate}
\end{definition}

\begin{definition}[Weak convergence of pyramids]
\label{def:pyramid-weak-convergence}
Let $\mathcal C$ be a $\Box$-metrized category. A sequence of pyramids $\mathcal P_n$ in $\mathcal C$ is said to \emph{converge weakly} to a pyramid $\mathcal P$ if it satisfies the following conditions.
\begin{enumerate}[label=\textup{(\roman*)}]
\item For every $A\in\mathcal P$, we have $\Box_{\mathcal C}(A,\mathcal P_n)\to0$.
\item For every $A\notin\mathcal P$, we have $\liminf_{n\to\infty}\Box_{\mathcal C}(A,\mathcal P_n)>0$.
\end{enumerate}
This is sequential Painlev\'e--Kuratowski convergence of the object sets. It is also called weak Hausdorff convergence in mm-space theory.
\end{definition}

\begin{definition}[Domination refinement]
\label{def:domination-refinement}
A $\Box$-metrized category $\mathcal C$ is said to have \emph{domination refinement} if, whenever $A\preceq\overline A$ and $\overline A_n\to\overline A$ with respect to $\Box_{\mathcal C}$, there exist $A_n\in\mathcal C$ such that
\[A_n\preceq\overline A_n,\qquad A_n\longrightarrow A\quad\text{with respect to $\Box_{\mathcal C}$.}\]
\end{definition}

\begin{definition}[Pyramidal adjunction]
\label{def:pyramidal-adjunction}
An adjunction between two $\Box$-metrized categories $\mathcal B$ and $\mathcal A$,
\[L\colon\mathcal B\rightleftarrows\mathcal A\colon R,\qquad L\dashv R,\]
is called a \emph{pyramidal adjunction} if the following three conditions hold.
\begin{enumerate}[label=\textup{(\roman*)}]
\item The unit $\operatorname{id}_{\mathcal B}\to R L$ is a natural isomorphism.
\item The functor $L$ is box-isometric, and $R$ is box-nonexpansive. Explicitly,
\[
\Box_{\mathcal A}(L(b),L(b'))=\Box_{\mathcal B}(b,b')
\qquad(b,b'\in\mathcal B),
\]
and
\[
\Box_{\mathcal B}(R(A),R(A'))\leq\Box_{\mathcal A}(A,A')
\qquad(A,A'\in\mathcal A).
\]
\item The category $\mathcal A$ has domination refinement.
\end{enumerate}
\end{definition}

\begin{proposition}
\label{prop:domination-refinement-descends}
Let $L\colon\mathcal B\rightleftarrows\mathcal A\colon R$ be an adjunction satisfying conditions~\textup{(i)} and~\textup{(ii)} of \Cref{def:pyramidal-adjunction}. If $\mathcal A$ has domination refinement, then $\mathcal B$ also has domination refinement.
\end{proposition}

\begin{proof}
Suppose that $b\preceq\overline b$ and $\overline b_n\to\overline b$ in $\mathcal B$. By domination refinement in $\mathcal A$, take $A_n$ such that
\[A_n\preceq L(\overline b_n),\qquad A_n\longrightarrow L(b).\]
Set $b_n\coloneqq R(A_n)$. Functoriality and the fact that the unit is an isomorphism give $b_n\preceq\overline b_n$, while
\[\Box_{\mathcal B}(b_n,b)\leq\Box_{\mathcal A}(A_n,L(b))\longrightarrow0.\]
This completes the proof.
\end{proof}

Every monoidal subfamily $\L$ considered here contains $\TB$. Specializing \cite[Lemma~5.4]{gds2} shows that $\L\circ\D$ has domination refinement.

For a pyramidal adjunction $L\colon\mathcal B\rightleftarrows\mathcal A\colon R$ and a pyramid $\mathcal P$ in $\mathcal B$, define
\begin{equation}
\label{eq:pyramid-transport}
L_{\#}\mathcal P\coloneqq\{A\in\mathcal A\mid R(A)\in\mathcal P\}.
\end{equation}

\begin{theorem}
\label{thm:abstract-pyramid-transport}
Let $L\colon\mathcal B\rightleftarrows\mathcal A\colon R$ be a pyramidal adjunction. Then the following statements hold.
\begin{enumerate}[label=\textup{(\arabic*)}]
\item The subcategory $L_{\#}\mathcal P$ is a pyramid in $\mathcal A$ and equals the downward closure of $L[\mathcal P]$.
\item For pyramids $\mathcal P_n,\mathcal P$ in $\mathcal B$,
\[
\mathcal P_n\longrightarrow\mathcal P
\quad\Longleftrightarrow\quad
L_{\#}\mathcal P_n\longrightarrow L_{\#}\mathcal P.
\]
\item The map $\mathcal P\mapsto L_{\#}\mathcal P$ is injective on pyramids of $\mathcal B$.
\end{enumerate}
\end{theorem}

\begin{proof}
For $L\dashv R$, the adjunction gives natural bijections
\[\operatorname{Hom}_{\mathcal A}(L(b),A)\simeq\operatorname{Hom}_{\mathcal B}(b,R(A)).\]
By the convention $A\preceq B$ when there is a morphism from $B$ to $A$, the left set is nonempty exactly when $A\preceq L(b)$, and the right set is nonempty exactly when $R(A)\preceq b$. Hence
\begin{equation}
\label{eq:adjunction-order}
A\preceq L(b)\quad\Longleftrightarrow\quad R(A)\preceq b.
\end{equation}
The unit and counit have components $\eta_b\colon b\to R(L(b))$ and $\varepsilon_A\colon L(R(A))\to A$. Since $\varepsilon_A$ is a morphism from $L(R(A))$ to $A$, our order convention gives $A\preceq L(R(A))$.
The unit, functoriality, and \Cref{eq:adjunction-order} show that $L_{\#}\mathcal P$ is nonempty, downward closed, and directed. If $A_n\in L_{\#}\mathcal P$ and $A_n\to A$, then $R(A_n)\to R(A)$, which proves box-closedness. If $A\in L_{\#}\mathcal P$, the counit relation just established gives $A\preceq L(R(A))$, and (1) follows.

Suppose that $\mathcal P_n\to\mathcal P$. For $A\in L_{\#}\mathcal P$, take $b_n\in\mathcal P_n$ with $b_n\to R(A)$. Apply domination refinement to $A\preceq L(R(A))$ and $L(b_n)\to L(R(A))$ to obtain
\[A_n\preceq L(b_n),\qquad A_n\longrightarrow A.\]
Then $A_n\in L_{\#}\mathcal P_n$, which proves the inner condition. The outer condition follows from the box-nonexpansiveness of $R$.

Conversely, suppose that $L_{\#}\mathcal P_n\to L_{\#}\mathcal P$. For $b\in\mathcal P$, take $A_n\in L_{\#}\mathcal P_n$ with $A_n\to L(b)$ and apply $R$ to obtain the inner condition. If $b_k\in\mathcal P_{n(k)}$ and $b_k\to b$, then $L(b_k)\to L(b)$. The outer condition gives $L(b)\in L_{\#}\mathcal P$, and therefore $b\in\mathcal P$. This proves (2).

If $L_{\#}\mathcal P=L_{\#}\mathcal Q$ and $b\in\mathcal P$, then $L(b)\in L_{\#}\mathcal Q$. The unit gives $b\simeq R(L(b))\in\mathcal Q$. The reverse inclusion is identical, which proves (3).
This completes the proof.
\end{proof}

\begin{proposition}
\label{prop:pyramidal-adjunction-composition}
For $i=1,2$, let
\[L_i\colon\mathcal C_{i-1}\rightleftarrows\mathcal C_i\colon R_i\]
be pyramidal adjunctions. Then the composite adjunction $L_2L_1\dashv R_1R_2$ is also pyramidal, and
\[(L_2L_1)_{\#}\mathcal P=(L_2)_{\#}(L_1)_{\#}\mathcal P.\]
The same statement holds for every finite composite.
\end{proposition}

\begin{proof}
The unit, the box-isometry of the left adjoint, and the box-nonexpansiveness of the right adjoint are preserved under composition. The category $\mathcal C_2$, being the codomain of the second adjunction, has domination refinement. The displayed identity follows because both sides impose the same condition $R_1R_2(A)\in\mathcal P$.
This completes the proof.
\end{proof}

\begin{definition*}[Associated pyramid in a $\Box$-metrized category]
For a $\Box$-metrized category $\mathcal C$ and an object $C$ of $\mathcal C$, write
\[\mathcal P_{\mathcal C}(C)\coloneqq\{A\in\mathcal C\mid A\preceq C\}.\]
When this lower set is a pyramid, it is called the \emph{associated pyramid} of $C$.
\end{definition*}

\begin{proposition}
\label{prop:pyramidal-compactification-pullback}
Under the assumptions of \Cref{thm:abstract-pyramid-transport}, for every $b\in\mathcal B$,
\[L_{\#}\mathcal P_{\mathcal B}(b)=\mathcal P_{\mathcal A}(L(b)).\]
Whenever these lower sets are pyramids, for $b_n\in\mathcal B$ and a pyramid $\mathcal P$ in $\mathcal B$,
\[
\mathcal P_{\mathcal B}(b_n)\longrightarrow\mathcal P
\quad\Longleftrightarrow\quad
\mathcal P_{\mathcal A}(L(b_n))\longrightarrow L_{\#}\mathcal P.
\]
\end{proposition}

\begin{proof}
By \Cref{eq:adjunction-order},
\[
A\in L_{\#}\mathcal P_{\mathcal B}(b)
\quad\Longleftrightarrow\quad R(A)\preceq b
\quad\Longleftrightarrow\quad A\preceq L(b).
\]
The equivalence of weak convergence follows from \Cref{thm:abstract-pyramid-transport}.
This completes the proof.
\end{proof}

\subsection{Specialization to qm-spaces}
We apply \Cref{def:categorical-pyramid,def:pyramid-weak-convergence} to $\mathcal C=\Xp$ and $\mathcal C=\Lup\circ\D$ and use these specializations to define pyramids and weak convergence in the two categories. On the qm-space side, the resulting set of pyramids is $\Pi^+$.

The domination-refinement property of $\Lup\circ\D$ follows by specializing the preceding input to $\L=\Lup$. It approximates an object dominated by a limit by objects below the corresponding terms of a Box-convergent sequence.

\begin{proposition}
\label{prop:asymmetric-domination-refinement}
The category $\Xp$ has domination refinement with respect to $\Box$.
\end{proposition}

\begin{proof}
The unit isomorphism, box-isometry of $\operatorname{Rep}^{+}$, and box-nonexpansiveness of $\operatorname{Rec}^{+}$ follow from \Cref{thm:adjunction-direction,def:pullback-distances,prop:right-adjoint-box-nonexpansive}. Together with domination refinement for $\Lup\circ\D$ from \cite[Lemma~5.4]{gds2}, these are exactly the hypotheses of \Cref{prop:domination-refinement-descends}.
This completes the proof.
\end{proof}

For a pyramid $\mathcal P$ of qm-spaces, set
\[
(\operatorname{Rep}^{+})_{\#}\mathcal P
\coloneqq\{Z\in\Lup\circ\D\mid\operatorname{Rec}^{+}(Z)\in\mathcal P\}.
\]
This is the specialization of \Cref{eq:pyramid-transport} to $\operatorname{Rep}^{+}\dashv\operatorname{Rec}^{+}$. For $Z\in\Lup\circ\D$, write
\[
\mathcal P_{\Lup\circ\D}(Z)\coloneqq\{W\in\Lup\circ\D\mid W\preceq Z\}.
\]
Together with $\mathcal P(X)$, this is the specialization of the associated lower set in \Cref{prop:pyramidal-compactification-pullback} to the two categories.

\begin{theorem}
\label{thm:qm-pyramid-transport}
Let $\mathcal P$ be a pyramid of qm-spaces. Then $(\operatorname{Rep}^{+})_{\#}\mathcal P$ is a pyramid in $\Lup\circ\D$. It equals the downward closure of $\operatorname{Rep}^{+}[\mathcal P]$. For pyramids $\mathcal P_n,\mathcal P$ of qm-spaces,
\[
\mathcal P_n\longrightarrow\mathcal P
\quad\Longleftrightarrow\quad
(\operatorname{Rep}^{+})_{\#}\mathcal P_n\longrightarrow(\operatorname{Rep}^{+})_{\#}\mathcal P,
\]
and the map $\mathcal P\mapsto(\operatorname{Rep}^{+})_{\#}\mathcal P$ is injective. Moreover, $\mathcal P(X)$ is a pyramid and
\begin{equation}
\label{eq:qm-associated-pyramid-pullback}
(\operatorname{Rep}^{+})_{\#}\mathcal P(X)
=\mathcal P_{\Lup\circ\D}(\operatorname{Rep}^{+}(X)).
\end{equation}
\end{theorem}

The pyramid $\mathcal P(X)$ in \Cref{thm:qm-pyramid-transport} is called the \emph{associated pyramid} of $X$.

\begin{proof}[Proof of \Cref{thm:qm-pyramid-transport}]
We verify the three conditions in \Cref{def:pyramidal-adjunction}. By \Cref{thm:adjunction-direction}, the unit of $\operatorname{Rep}^{+}\dashv\operatorname{Rec}^{+}$ is an isomorphism. By \Cref{def:pullback-distances,prop:right-adjoint-box-nonexpansive}, $\operatorname{Rep}^{+}$ is box-isometric and $\operatorname{Rec}^{+}$ is box-nonexpansive.

Specializing \cite[Lemma~5.4]{gds2} to $\L=\Lup$ shows that the codomain $\Lup\circ\D$ has domination refinement. Thus this adjunction is pyramidal. Pyramid transport, equivalence of weak convergence, and injectivity follow from \Cref{thm:abstract-pyramid-transport}.

Box-closedness of the associated lower sets follows separately on the two sides. On the $\Lup\circ\D$ side, apply \Cref{prop:domination-box-closedness}. On the qm-space side, suppose that $Y_n\preceq X$ and $Y_n\to Y$. Then
\[
\operatorname{Rep}^{+}(Y_n)\preceq\operatorname{Rep}^{+}(X),\qquad
\operatorname{Rep}^{+}(Y_n)\longrightarrow\operatorname{Rep}^{+}(Y).
\]
Applying \Cref{prop:domination-box-closedness} again gives $\operatorname{Rep}^{+}(Y)\preceq\operatorname{Rep}^{+}(X)$. By the two-object Hom-set identification following \Cref{thm:adjunction-direction}, the corresponding qm-space morphism is $X\to Y$, so $Y\preceq X$. Thus, the lower sets on both sides are pyramids. Finally, \Cref{eq:qm-associated-pyramid-pullback} follows from \Cref{prop:pyramidal-compactification-pullback}.
This completes the proof.
\end{proof}

\section{Weak compactness of pyramids}
\label{sec:weak-compactness}

A limit of closed sets need not remain downward closed or directed. Domination refinement supplies the lower objects, while compact refinement of common upper bounds supplies directedness. Together they give weak limits and sequential compactness.

\begin{theorem}
\label{thm:gds-pyramid-weak-limit-and-sequential-compactness}
Let $\mathcal P_n$ be pyramids in $\Lup\circ\D$. If $\mathcal P_n$ converges as a sequence of closed sets in the sense of sequential Painlev\'e--Kuratowski convergence to a box-closed set $\mathcal E\subset\Lup\circ\D$, then $\mathcal E$ is a pyramid. Moreover, every sequence of pyramids in $\Lup\circ\D$ has a subsequence that converges weakly to a pyramid.
\end{theorem}

The refinement argument proving this theorem appears in \Cref{subsec:appendix-compactness-refinements}.

\begin{corollary}
\label{cor:pyramid-weak-limit-and-sequential-compactness}
Let $\mathcal P_n$ be pyramids in $\Xp$. If $\mathcal P_n$ converges as a sequence of closed sets to a box-closed set $\mathcal E\subset\Xp$ in the sense of sequential Painlev\'e--Kuratowski convergence, then $\mathcal E$ is a pyramid. Moreover, every sequence of pyramids in $\Xp$ has a subsequence that converges weakly to a pyramid.
\end{corollary}

\begin{proof}
Set $L\coloneqq\operatorname{Rep}^{+}(-)$ and $R\coloneqq\operatorname{Rec}^{+}(-)$, and for every closed set $\mathcal F\subset\Xp$, define
\[L_{\#}\mathcal F\coloneqq\{X\in\Lup\circ\D\mid R(X)\in\mathcal F\}.\]
Since $R$ is box-nonexpansive, $L_{\#}\mathcal F$ is box-closed whenever $\mathcal F$ is box-closed. We apply the inner and outer arguments from \Cref{thm:abstract-pyramid-transport} to $\mathcal P_n\to\mathcal E$.

For the inner condition, take $X\in L_{\#}\mathcal E$. Applying the inner condition for $\mathcal P_n\to\mathcal E$ to $R(X)\in\mathcal E$ gives $A_n\in\mathcal P_n$ with $A_n\to R(X)$. Since $L$ is box-isometric, $L(A_n)\to LR(X)$. Apply \cite[Lemma~5.4]{gds2} to this convergence and the counit domination $X\preceq LR(X)$. We obtain $\widetilde X_n\in\Lup\circ\D$ such that
\[
\widetilde X_n\preceq L(A_n),\qquad \widetilde X_n\longrightarrow X.
\]
The unit is an isomorphism, so $RL(A_n)\simeq A_n$, and functoriality gives $R(\widetilde X_n)\preceq RL(A_n)$. Since $\mathcal P_n$ is downward closed, $R(\widetilde X_n)\in\mathcal P_n$. Hence $\widetilde X_n\in L_{\#}\mathcal P_n$, proving the inner condition.

For the outer condition, suppose that $X_k\in L_{\#}\mathcal P_{n(k)}$ and $X_k\to X$. The box-nonexpansiveness of $R$ gives $R(X_k)\to R(X)$, while the definition gives $R(X_k)\in\mathcal P_{n(k)}$. The outer condition for $\mathcal P_n\to\mathcal E$ yields $R(X)\in\mathcal E$, and therefore $X\in L_{\#}\mathcal E$. Thus $L_{\#}\mathcal P_n\to L_{\#}\mathcal E$.

The same theorem shows that $L_{\#}\mathcal P_n$ is a pyramid, so \Cref{thm:gds-pyramid-weak-limit-and-sequential-compactness} implies that $L_{\#}\mathcal E$ is also a pyramid. The fact that the unit is an isomorphism gives
\[A\in\mathcal E\quad\Longleftrightarrow\quad L(A)\in L_{\#}\mathcal E.\]
This equivalence implies that $\mathcal E$ is nonempty, downward closed, and box-closed. If $X\in L_{\#}\mathcal E$ is a common upper bound of $L(A)$ and $L(B)$, then $R(X)\in\mathcal E$ is a common upper bound of $A$ and $B$. Therefore, $\mathcal E$ is a pyramid.

Next, take a sequence of pyramids $\mathcal P_n$ in $\Xp$. By \Cref{thm:gds-pyramid-weak-limit-and-sequential-compactness}, after passing to a subsequence, $L_{\#}\mathcal P_n$ converges weakly to a pyramid $\mathcal Q$ in $\Lup\circ\D$. For $X\in\mathcal Q$, the inner condition gives $X_n\in L_{\#}\mathcal P_n$ such that $X_n\to X$. We have $LR(X_n)\in L_{\#}\mathcal P_n$, and the box-isometry of $L$ and the box-nonexpansiveness of $R$ give $LR(X_n)\to LR(X)$. The outer condition then gives $LR(X)\in\mathcal Q$. Set
\[\mathcal P\coloneqq\{A\in\Xp\mid L(A)\in\mathcal Q\}.\]
Since $L$ is box-isometric, $\mathcal P$ is box-closed. Taking any $X\in\mathcal Q$ gives $R(X)\in\mathcal P$, so $\mathcal P$ is nonempty, and its downward closedness follows from the functoriality of $L$ and the downward closedness of $\mathcal Q$. For $A,B\in\mathcal P$, take $X\in\mathcal Q$ such that $L(A),L(B)\preceq X$. Then $LR(X)\in\mathcal Q$ and $A,B\preceq R(X)$, so $R(X)\in\mathcal P$ is a common upper bound. Therefore, $\mathcal P$ is a pyramid. Moreover, if $X\in\mathcal Q$, then $LR(X)\in\mathcal Q$, so $R(X)\in\mathcal P$. Conversely, if $R(X)\in\mathcal P$, then $LR(X)\in\mathcal Q$, and the domination $X\preceq LR(X)$ given by the counit implies $X\in\mathcal Q$. Thus, $\mathcal Q=L_{\#}\mathcal P$. By the equivalence of weak convergence in \Cref{thm:abstract-pyramid-transport}, this subsequence converges weakly to $\mathcal P$.
This completes the proof.
\end{proof}

\begin{theorem}
\label{thm:associated-pyramid-sequential-density}
Let $(\mathcal C,\Box)$ be a $\Box$-metrized category satisfying the following two conditions.
\begin{enumerate}[label=\textup{(\roman*)}]
\item Every pyramid in $\mathcal C$ is separable with respect to $\Box$.
\item The associated lower set $\mathcal P_{\mathcal C}(C)$ is a pyramid for every object $C$ of $\mathcal C$.
\end{enumerate}
Then, for every pyramid $\mathcal P$ in $\mathcal C$, there exists a sequence of objects $Z_n\in\mathcal P$ such that $\mathcal P_{\mathcal C}(Z_n)$ converges weakly to $\mathcal P$. Consequently, associated pyramids are sequentially dense in the space of all pyramids in $\mathcal C$.
\end{theorem}

\begin{proof}
Take a box-dense sequence $Y_1,Y_2,\ldots$ in $\mathcal P$. Set $Z_1\coloneqq Y_1$. For $n\geq2$, use directedness inductively to choose $Z_n\in\mathcal P$ such that
\[Z_{n-1}\preceq Z_n,\qquad Y_n\preceq Z_n.\]
Then $Y_j\preceq Z_n$ for $j\leq n$, and $\mathcal P_{\mathcal C}(Z_n)\subset\mathcal P$.

For $A\in\mathcal P$ and $\varepsilon>0$, choose $N$ such that $\Box(A,Y_N)<\varepsilon$. For every $n\geq N$,
\[\Box(A,\mathcal P_{\mathcal C}(Z_n))<\varepsilon.\]
On the other hand, if $A\notin\mathcal P$, then the box-closedness of $\mathcal P$ and the preceding inclusion give
\[\Box(A,\mathcal P_{\mathcal C}(Z_n))\geq\Box(A,\mathcal P)>0.\]
Thus the inner and outer conditions for weak convergence hold.
This completes the proof.
\end{proof}

\section{Ordered finite measurements}
\label{sec:ordered-measurements}

Weak convergence alone provides neither a concrete compact metric nor a Lipschitz estimate for the associated-pyramid map. We instead push the measure forward by finite tuples of one-sided observables. Keeping the coordinates ordered retains the directed endpoint information and metrizes weak convergence.

For a positive integer $N$ and a real number $R>0$, let $\M(N,R)$ denote the set of all Borel probability measures on $[-R,R]^N$, equipped with the Prokhorov distance $\dpr$ induced by the $\ell^\infty$ distance. Set $b_R(t)\coloneqq\max\{-R,\min\{t,R\}\}$.

\begin{definition}[$(N,R)$-measurement of a qm-space]
\label{def:ordered-measurement}
The \emph{$(N,R)$-measurement} of a qm-space $X$ is defined by
\[
\M(X;N,R)\coloneqq\{(b_R\circ f_1,\ldots,b_R\circ f_N)_*\mu_X\mid(f_1,\ldots,f_N)\in\lipplus(X)^N\}\subset\M(N,R).
\]
\end{definition}

\begin{definition}[$(N,R)$-measurement of a pyramid of qm-spaces]
\label{def:qm-pyramid-ordered-measurement}
For a pyramid $\mathcal P$ of qm-spaces, set
\[\M(\mathcal P;N,R)\coloneqq\bigcup_{X\in\mathcal P}\M(X;N,R).\]
\end{definition}

\begin{lemma}
\label{lem:ordered-measurement-stability}
For qm-spaces $X,Y$, we have
\[
\haus{\dpr}(\M(X;N,R),\M(Y;N,R))\leq N\dconc(X,Y).
\]
\end{lemma}

A proof of this lemma is given in \Cref{subsec:appendix-qm-measurement-specialization}.

\subsection{The measurement-induced pyramid metric}

\begin{definition}[Pyramid metric for qm-spaces]
\label{def:measurement-pyramid-metric}
For two pyramids $\mathcal P,\mathcal Q$ of qm-spaces, set
\begin{equation}
\label{eq:measurement-pyramid-metric}
d_{\Pi}^{\M}(\mathcal P,\mathcal Q)
\coloneqq\sum_{N=1}^{\infty}\frac{1}{N\cdot2^N}
\haus{\dpr}\left(
\M(\mathcal P;N,N),
\M(\mathcal Q;N,N)
\right).
\end{equation}
This is the measurement-set weighting of the classical pyramid metric \cite[Definition~6.17 and Theorem~6.18]{shioya2024sgc}, with mm-space measurements replaced by the $(N,R)$-measurements above and every weight multiplied by two.
\end{definition}

\begin{theorem}
\label{thm:pyramid-ordered-measurement-criterion}
\label{thm:measurement-pyramid-metric}
\Cref{eq:measurement-pyramid-metric} defines a metric on the set of all pyramids of qm-spaces. For pyramids $\mathcal P_n,\mathcal P$ of qm-spaces, the following are equivalent.
\begin{enumerate}[label=\textup{(\roman*)}]
\item $\mathcal P_n$ converges weakly to $\mathcal P$.
\item For every positive integer $N$ and every real number $R>0$,
\[\haus{\dpr}(\M(\mathcal P_n;N,R),\M(\mathcal P;N,R))\longrightarrow0.\]
\item $d_{\Pi}^{\M}(\mathcal P_n,\mathcal P)\to0$.
\end{enumerate}
\end{theorem}

We prove this theorem in \Cref{subsec:appendix-qm-measurement-specialization}.

\section{Pyramid compactification of qm-spaces}
\label{sec:compactification}

We now combine directed-distance recovery, pyramid transport, weak compactness, and the metric supplied by finite measurements. This yields compactness, density of associated pyramids, the Lipschitz estimate, and the topological embedding.

\begin{proof}[Completion of the proof of \Cref{thm:intro-qm-pyramidal-compactification}]
By \Cref{cor:pyramid-weak-limit-and-sequential-compactness}, every sequence in $\Pi^+$ has a weakly convergent subsequence. The metric $d_{\Pi}^{\M}$ from \Cref{thm:pyramid-ordered-measurement-criterion} metrizes weak convergence. Therefore, $(\Pi^+,d_{\Pi}^{\M})$ is sequentially compact.

A Cauchy sequence with a convergent subsequence converges, so the space is complete. If the space were not totally bounded, there would be an $\varepsilon>0$ for which no finite $\varepsilon$-net exists. Choosing each point outside the union of the $\varepsilon$-balls centered at the preceding points would produce an $\varepsilon$-separated sequence. Such a sequence has no convergent subsequence, contradicting sequential compactness. Thus the space is totally bounded and complete, and hence compact.

By \Cref{cor:pullback-box-polish}, $(\Xp,\Box)$ is complete and separable, so every pyramid in $\Xp$ is separable with respect to $\Box$. By \Cref{thm:qm-pyramid-transport}, the lower set associated with every qm-space is a pyramid. Applying \Cref{thm:associated-pyramid-sequential-density} to any $\mathcal P\in\Pi^+$ gives qm-spaces $X_n\in\mathcal P$ such that $\mathcal P(X_n)$ converges weakly to $\mathcal P$. Thus, the image of the associated-pyramid map is dense.

If $B\preceq A$, every one-sided $1$-Lipschitz function on $B$ pulls back along a domination to one on $A$. The reverse inclusion uses $A\in\mathcal P(A)$. Hence, for every $A\in\Xp$, every positive integer $N$, and every $R>0$,
\[\M(\mathcal P(A);N,R)=\M(A;N,R).\]

For $X,Y\in\Xp$, set $\delta\coloneqq\dconc(X,Y)$. The preceding identity and \Cref{lem:ordered-measurement-stability} with $R=N$ show that the $N$th Hausdorff term in \Cref{eq:measurement-pyramid-metric} satisfies
\[\haus{\dpr}\left(\M(X;N,N),\M(Y;N,N)\right)\leq N\delta.\]
Therefore,
\[
d_{\Pi}^{\M}(\mathcal P(X),\mathcal P(Y))
\leq\sum_{N=1}^{\infty}\frac{N\delta}{N\cdot2^N}
=\delta\sum_{N=1}^{\infty}2^{-N}
=\delta.
\]

The topology induced on the image is determined by weak convergence. By \Cref{thm:measurement-pyramid-metric}, $d_{\Pi}^{\M}$ metrizes weak convergence, and by \Cref{thm:qm-pyramid-transport}, $(\operatorname{Rep}^{+})_{\#}$ preserves and reflects weak convergence. Combining these facts with the topological embedding of $\Lup\circ\D$ into its pyramid space from \cite[Proposition~7.3]{gds2} and \Cref{eq:qm-associated-pyramid-pullback} gives, for qm-spaces $X_n,X$,
\[
\mathcal P(X_n)\longrightarrow\mathcal P(X)\text{ weakly}
\quad\Longleftrightarrow\quad
\dconc(X_n,X)\longrightarrow0.
\]
Injectivity follows from \Cref{thm:qm-pyramid-transport} and the injectivity of the embedding on $\Lup\circ\D$. Thus, the associated-pyramid map is a topological embedding. This completes the proof.
\end{proof}

\section{Consistency with the symmetric theory and applications}
\label{sec:symmetric-applications}

We now apply the abstract transport of \Cref{sec:abstract-transport} to symmetric mm-spaces. The three adjunctions are summarized in \Cref{fig:three-layer-adjunction}.

We defined the symmetric functors in \Cref{def:mm-space-representation,rem:symmetric-reconstruction-bridge}. They are denoted by $\operatorname{Rep}(-)$ and $\operatorname{Rec}(-)$, respectively, and leave the underlying maps unchanged on morphisms.

By \Cref{eq:symmetric-recovery}, $\operatorname{Rep}(-)$ is a functor from $\X$ to $\lipone(\R)\circ\D$. The definition of a gd-set shows that $\operatorname{Rec}(-)$ is a functor in the reverse direction.

\begin{proposition}
\label{prop:symmetric-adjunction}
There is a natural adjunction $\operatorname{Rep}\dashv\operatorname{Rec}$. Its unit is an isomorphism, and $\operatorname{Rep}$ is fully faithful.
\end{proposition}
The adjunction statement is proved in \Cref{subsec:appendix-symmetric-proofs}.

Define the symmetrization functor $\operatorname{Sym}(-)\colon\Xp\to\X$ by $\operatorname{Sym}(Y)\coloneqq(Y,d_Y^{\mathrm s},\mu_Y)$.

\begin{proposition}
\label{prop:sym-adjunction}
For every $X\in\X$ and $Y\in\Xp$, there is a natural bijection that leaves the underlying maps unchanged,
\[\operatorname{Hom}_{\Xp}(I(X),Y)\simeq\operatorname{Hom}_{\X}(X,\operatorname{Sym}(Y)).\]
Thus $I\dashv\operatorname{Sym}$, and its unit is an isomorphism.
\end{proposition}
The proof of this proposition appears in \Cref{subsec:appendix-symmetric-proofs}.

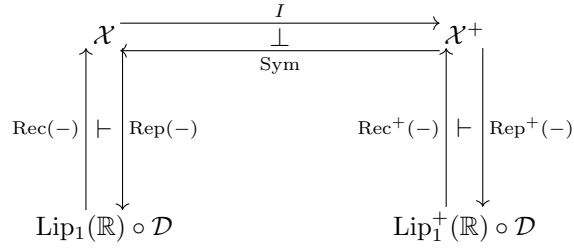
\begin{figure}[ht]
\centering
\begin{tikzcd}[
  column sep=8.0em,
  row sep=6.0em,
  cells={nodes={inner xsep=.35ex,inner ysep=.35ex}}
]
\X
  \arrow[r,shift left=1.2ex,"I"]
  \arrow[r,phantom,"\bot"{description,fill=none,inner sep=0pt}]
  \arrow[d,xshift=1.6ex,"\operatorname{Rep}(-)"]
  \arrow[d,phantom,"\vdash"{description,fill=none,inner sep=0pt}]
& \Xp
  \arrow[l,shift left=1.2ex,"\operatorname{Sym}"]
  \arrow[d,xshift=1.6ex,"\operatorname{Rep}^{+}(-)"]
  \arrow[d,phantom,"\vdash"{description,fill=none,inner sep=0pt}]
\\
\lipone(\R)\circ\D
  \arrow[u,xshift=-1.6ex,"\operatorname{Rec}(-)"]
& \Lup\circ\D
  \arrow[u,xshift=-1.6ex,"\operatorname{Rec}^{+}(-)"]
\end{tikzcd}
\caption{Adjunctions connecting the symmetric, directed, and gd-set layers.}
\label{fig:three-layer-adjunction}
\end{figure}

\begin{corollary}
\label{cor:figure-adjunctions-pyramidal}
All three adjunctions in \Cref{fig:three-layer-adjunction} are pyramidal. Therefore, every finite composite of left adjoints in the diagram is also a pyramidal adjunction.
\end{corollary}
A proof of the corollary is given in \Cref{subsec:appendix-symmetric-proofs}.

By \Cref{prop:right-adjoint-box-nonexpansive}, $(\operatorname{Rep}\circ\operatorname{Rec})(-)$ is box-nonexpansive. Apply the idempotent-image argument of \Cref{cor:pullback-box-polish} with $\operatorname{Rep}$ and $\operatorname{Rec}$. There, replace $\operatorname{Rep}^{+}$, $\operatorname{Rec}^{+}$, and $\Lup\circ\D$ by $\operatorname{Rep}$, $\operatorname{Rec}$, and $\lipone(\R)\circ\D$, respectively. Using \cite[Theorem~4.10]{gds2}, we see that $(\X,\Box)$ is complete and separable.

By \Cref{cor:figure-adjunctions-pyramidal,prop:pyramidal-adjunction-composition}, the transported pyramids of a pyramid $\mathcal P$ in $\X$ are
\[
I_{\#}\mathcal P=\{A\in\Xp\mid\operatorname{Sym}(A)\in\mathcal P\},\qquad
(\operatorname{Rep}^{+}\circ I)_{\#}\mathcal P=(\operatorname{Rep}^{+})_{\#}(I_{\#}\mathcal P),
\]
and \Cref{thm:abstract-pyramid-transport} shows that both transports preserve and reflect weak convergence.

\begin{remark}
The motivating application is the \emph{beta model with the forward Funk distance}. For $n\geq1$ and $\beta>0$, define the probability measure on the Euclidean unit ball $\mathbb B^n$ by \cite[Equation~(1.2)]{finsler-poincare-beta-balls}
\[
d\mu_{n,\beta}(x)
=\frac{\Gamma\left(\frac n2+\beta\right)}{\pi^{n/2}\Gamma(\beta)}
(1-|x|^2)^{\beta-1}\,dx.
\]
Following \cite[Theorem~3.1 and Equations~(3.4)--(3.5)]{finsler-poincare-beta-balls}, let $d_K$ be the Klein distance and put $\phi(x)\coloneqq-\frac12\log(1-|x|^2)$. The forward Funk distance then satisfies
\[
d_F(x,y)=d_K(x,y)+\phi(y)-\phi(x),
\]
and its symmetrization is $d_F^{\mathrm s}(x,y)=d_K(x,y)+|\phi(y)-\phi(x)|$. In particular, \cite[Equation~(1.1)]{finsler-poincare-beta-balls} gives, as $r\uparrow1$,
\[
d_F(0,r\theta)=-\log(1-r)\longrightarrow+\infty,
\qquad
d_F(r\theta,0)=\log(1+r)\leq\log2.
\]
Thus symmetrization alone does not retain the radial asymmetry. The pyramid limits and three-phase theorem for this model are proved in the sister paper \cite[Theorem~5.1]{finsler-poincare-beta-balls}, which uses the compactification and adjoint transport developed here. The proofs and main results of the present paper do not depend on the sister paper. We recall its formulas here to specify the application context.
\end{remark}

\appendix

\section{Details of compactness and box estimates}
\label{sec:appendix-compactness-box}

This appendix contains the gluing, refinement, and symmetric-side estimates deferred from the main line.

\subsection{Box-closedness of domination}
\label{subsec:appendix-domination-closedness}
\begin{proof}[Proof of \Cref{prop:domination-box-closedness}]
Choose numbers $\varepsilon_n\downarrow0$ such that
\[\Box(X_n,X)<\varepsilon_n,\qquad \Box(Y_n,Y)<\varepsilon_n,\]
and choose couplings
\[
\alpha_n\in\Tplan(\mu_X,\mu_{X_n}),\qquad
\beta_n\in\Tplan(\mu_{Y_n},\mu_Y),
\]
and closed sets $A_n\subset X\times X_n$ and $B_n\subset Y_n\times Y$ such that
\[\alpha_n(A_n),\beta_n(B_n)\geq1-\varepsilon_n\]
and the two Hausdorff errors in \Cref{eq:gds-box} are at most $\varepsilon_n$.
Let $u_n\colon X_n\to Y_n$ be a domination. Apply the gluing of transport plans twice
\cite[Definition~4.1]{shioya2024sgc} to $\alpha_n$, the graph coupling
$(\operatorname{id}_{X_n},u_n)_*\mu_{X_n}$, and $\beta_n$. This gives a probability measure
$\eta_n$ on $X\times X_n\times Y_n\times Y$. Its endpoint marginal $\pi_n$ is a coupling
of $\mu_X$ and $\mu_Y$. By compactness of the set of couplings with fixed marginals
\cite[Lemma~2.10]{gds1}, after passing to a subsequence we have
\[\pi_n\longrightarrow\pi\in\Tplan(\mu_X,\mu_Y).\]
Write $R\coloneqq\supp\pi$.

Fix $g\in F_Y$. The Hausdorff estimate on $B_n$ gives
$g_n\in\overline{F_{Y_n}}$ whose values differ from those of $g$ by at most
$\varepsilon_n$ on $B_n$. Since $u_n$ is a domination,
$g_n\circ u_n\in\overline{F_{X_n}}$. The Hausdorff estimate on $A_n$ then gives
$f_n\in\overline{F_X}$ whose values differ from those of $g_n\circ u_n$ by at most
$\varepsilon_n$ on $A_n$. Consequently,
\begin{equation}
\label{eq:exact-domination-glued-error}
|f_n(x)-g(y)|\leq2\varepsilon_n
\end{equation}
whenever $(x,x_n)\in A_n$, $y_n=u_n(x_n)$, and $(y_n,y)\in B_n$.
The set of such quadruples has $\eta_n$-measure at least $1-2\varepsilon_n$.

For every $(x,y)\in R$, weak convergence of $\pi_n$ and
$\varepsilon_n\to0$ yield quadruples satisfying these three conditions whose endpoints
converge to $(x,y)$. Indeed, the Portmanteau theorem
\cite[Lemma~1.13]{shioya2016mmg} gives
$\liminf_n\pi_n(U)\geq\pi(U)>0$ for every open neighborhood $U$ of $(x,y)$.
The endpoint marginal of the complement of the good quadruples has mass at most
$2\varepsilon_n$. Therefore, $U$ contains a good endpoint for all sufficiently large $n$.
Passing to a further subsequence and applying this observation
at one point of $R$ shows that $(f_n)$ is bounded at one point of $X$. The functions
$f_n$ are $1$-Lipschitz with respect to $d_{F_X}$, so separability and a diagonal
argument give a further subsequence converging pointwise to some
$f\in\overline{F_X}$. Applying \Cref{eq:exact-domination-glued-error} at an arbitrary
point of $R$ gives
\[f(x)=g(y)\qquad ((x,y)\in R).\]

Repeating this argument for each $g\in F_Y$ shows that, for
$(x,y),(x',y')\in R$,
\[d_{F_Y}(y,y')\leq d_{F_X}(x,x').\]
The first projection of $R$ is dense in $X$ because $\mu_X$ has full support.
The preceding inequality and the completeness of $Y$ extend $R$ to the graph of
a uniquely determined $1$-Lipschitz map $u\colon X\to Y$. Since $R$ is closed, it
is precisely this graph. The second marginal of $\pi$ gives $u_*\mu_X=\mu_Y$, and
the equality constructed above gives $F_Y\circ u\subset\overline{F_X}$. Thus $u$
is a domination and $Y\preceq X$.
This completes the proof.
\end{proof}

\subsection{Refinement of pyramid limits}
\label{subsec:appendix-compactness-refinements}
\begin{proof}[Proof of \Cref{thm:gds-pyramid-weak-limit-and-sequential-compactness}]
The following argument applies to every $\L\circ\D$ with $\TB\subset\L$. The theorem is its specialization to $\L=\Lup$.

Take $A\in\mathcal E$ and $B\preceq A$. By the inner condition, there is a sequence $A_n\in\mathcal P_n$ such that $A_n\to A$. By \cite[Lemma~5.4]{gds2}, there are $B_n\preceq A_n$ with $B_n\to B$. Since every $\mathcal P_n$ is downward closed, $B_n\in\mathcal P_n$, and the outer condition gives $B\in\mathcal E$. Thus $\mathcal E$ is downward closed.

Let $A,B\in\mathcal E$. By the inner condition, choose $A_n,B_n\in\mathcal P_n$ that box-converge to $A,B$, respectively. Since $\mathcal P_n$ is directed, there exists $\overline C_n\in\mathcal P_n$ such that $A_n,B_n\preceq\overline C_n$. By \cite[Lemma~5.6]{gds2}, there are objects $C_n$ and a subsequence $C_{n(k)}\to C$ such that
\[A_n\preceq C_n,\qquad B_n\preceq C_n,\qquad C_n\preceq\overline C_n.\]
Then $C_n\in\mathcal P_n$. Applying \Cref{prop:domination-box-closedness} twice gives $A,B\preceq C$, and the outer condition gives $C\in\mathcal E$. Thus $\mathcal E$ is directed.

The category $\L\circ\D$ has a one-point object $\mathbf 1$ dominated by every object. It is obtained by taking the family of all constant functions on a one-point set. Since $\TB\subset\L$, the family $\overline{F_X}$ contains all constant functions for each $X\in\L\circ\D$, and the constant map is a domination. Therefore, $\mathbf 1\in\mathcal P_n$. If $\mathbf 1\notin\mathcal E$, then
\[\Box(\mathbf 1,\mathcal P_n)=0,\]
contrary to the outer condition. Thus $\mathcal E$ is nonempty. It is box-closed by assumption, and hence \Cref{def:categorical-pyramid} shows that $\mathcal E$ is a pyramid.

The metric space $(\L\circ\D,\Box)$ is complete and separable \cite[Theorem~4.10]{gds2}. Every sequence of closed subsets of a complete separable metric space has a subsequence that converges to a closed set in the sense of sequential Painlev\'e--Kuratowski convergence \cite[Theorem~5.2.12]{beer1993topologies}\cite[Lemma~2.11]{gds2}. The preceding argument shows that this closed set is a pyramid. Specializing to $\L=\Lup$ proves the theorem. This completes the proof.
\end{proof}

\subsection{The symmetric adjunctions and their box estimates}
\label{subsec:appendix-symmetric-proofs}
\begin{proof}[Proof of \Cref{prop:symmetric-adjunction}]
The Hom-set verification in the proof of \Cref{thm:adjunction-direction} applies with $\lipplus$, $d^+$, and $\operatorname{Rec}^{+}$ replaced by $\lipone$, $d_{F_Z}$, and $\operatorname{Rec}$. Namely, for a measure-preserving map $u\colon X\to Z$,
\[
F_Z\circ u\subset\lipone(X)
\quad\Longleftrightarrow\quad
d_{F_Z}(u(x),u(x'))\leq d_X(x,x').
\]
This gives the natural bijection of Hom-sets. The unit is an isomorphism by \Cref{eq:symmetric-recovery}.
This completes the proof.
\end{proof}
\begin{proof}[Proof of \Cref{prop:sym-adjunction}]
A map $u\colon X\to Y$ is $1$-Lipschitz from $I(X)$ to $Y$ if and only if
\[d_Y(u(x),u(x'))\leq d_X(x,x'),\qquad d_Y(u(x'),u(x))\leq d_X(x,x')\]
hold simultaneously. This is equivalent to $d_Y^{\mathrm s}(u(x),u(x'))\leq d_X(x,x')$. The unit $X\to\operatorname{Sym}(I(X))$ is the identity isomorphism.
This completes the proof.
\end{proof}
\begin{proof}[Proof of \Cref{cor:figure-adjunctions-pyramidal}]
For $I\dashv\operatorname{Sym}$, the unit is an isomorphism by \Cref{prop:sym-adjunction}. The identity \Cref{eq:symmetric-directed-lipschitz} and the pullback definitions show that $I$ is box-isometric.

We prove the box-nonexpansiveness of $\operatorname{Sym}$. Let $A,B\in\Xp$, $\pi\in\Tplan(\mu_A,\mu_B)$, and let $S\subset A\times B$ be closed. Set
\[H_S\coloneqq\haus{\dinf{S}}(\lipplus(A)\circ\pr_1,\lipplus(B)\circ\pr_2).\]
Assume that $S\neq\emptyset$ and $H_S<+\infty$, since otherwise the estimate below is immediate. For $(a,b),(a',b')\in S$, $f\in\lipplus(A)$, and $\varepsilon>0$, choose $g\in\lipplus(B)$ whose uniform error from $f$ on $S$ is less than $H_S+\varepsilon$. Then
\[
f(a')-f(a)\leq g(b')-g(b)+2H_S+2\varepsilon
\leq d_B(b,b')+2H_S+2\varepsilon.
\]
Taking the supremum, letting $\varepsilon\downarrow0$, and interchanging $A$ and $B$ gives $\dis S\leq2H_S$. Moreover,
\[
\dis_{\mathrm s}S\coloneqq
\sup_{\substack{(a,b),(a',b')\in S}}
\bigl|d_A^{\mathrm s}(a,a')-d_B^{\mathrm s}(b,b')\bigr|
\leq\dis S,
\]
because the maximum of the two directed distances changes by at most the larger of their two errors. Apply \Cref{lem:asymmetric-box-upper-by-distortion} to $I(\operatorname{Sym}(A))$ and $I(\operatorname{Sym}(B))$. The box-isometry of $I$ gives
\[
\Box(\operatorname{Sym}(A),\operatorname{Sym}(B))
\leq\max\{1-\pi(S),\dis_{\mathrm s}S\}
\leq\max\{1-\pi(S),2H_S\}.
\]
Taking the infimum over $\pi$ and $S$ and using \Cref{def:pullback-distances,def:pullback-mm-box-distance,eq:gds-box} proves that $\operatorname{Sym}$ is box-nonexpansive. The codomain $\Xp$ has domination refinement by \Cref{prop:asymmetric-domination-refinement}.

For $\operatorname{Rep}\dashv\operatorname{Rec}$ and $\operatorname{Rep}^{+}\dashv\operatorname{Rec}^{+}$, the unit isomorphisms, box-isometry of the left adjoints, box-nonexpansiveness of the right adjoints, and domination refinement of the codomains follow from \Cref{prop:symmetric-adjunction,thm:adjunction-direction,def:pullback-distances,def:pullback-mm-box-distance,prop:right-adjoint-box-nonexpansive} and \cite[Lemma~5.4]{gds2}. Finite composites are covered by \Cref{prop:pyramidal-adjunction-composition}.
This completes the proof.
\end{proof}

\section{General theory and technical proofs for ordered measurements}
\label{sec:appendix-ordered-measurements}

This appendix develops the ordered-measurement theory for a general monoidal subfamily. It also supplies the technical proofs of closedness, comparison, Strassen, and Prokhorov estimates. The qm-space case used in the main text is obtained at the end.

\begin{definition*}[Ordered measurements of gd-sets]
For a gd-set $X$, define its $(N,R)$-measurement by the formula in \Cref{def:ordered-measurement}, with $\lipplus(X)$ replaced by $F_X$.
\end{definition*}

\begin{definition*}[Ordered measurements of families of gd-sets]
For a family $\mathcal E$ of gd-sets, set
\[
\M(\mathcal E;N,R)\coloneqq\bigcup_{X\in\mathcal E}\M(X;N,R).
\]
\end{definition*}
For a pyramid, this union ranges over all its objects, including every representative isomorphic to one of its objects. Indeed, isomorphic objects dominate one another, so downward closure includes them both. The measurement of a single gd-set need not be invariant under isomorphism.

\begin{definition*}[Coordinate gd-set]
For a Borel probability measure $\nu$ on $\R^N$, set
\[\mathsf B_N(\nu)\coloneqq(\supp\nu,\{\pr_1,\ldots,\pr_N\},\nu).\]
\end{definition*}

\begin{definition*}[Unordered measurement space]
Write
\[\DM(N,R)\coloneqq\{[\mathsf B_N(\nu)]\mid\nu\in\M(N,R)\}.\]
\end{definition*}

\begin{definition*}[Coordinate-order-forgetting map]
The map that forgets the order and repetitions of the coordinates is
\[
q_{N,R}\colon\M(N,R)\longrightarrow\DM(N,R),\qquad
q_{N,R}(\nu)\coloneqq[\mathsf B_N(\nu)].
\]
\end{definition*}

\begin{definition*}[Unordered measurements of a gd-set]
For a gd-set $X$, set
\[
\DM(X;N,R)\coloneqq\{q_{N,R}((b_R\circ f_1,\ldots,b_R\circ f_N)_*\mu_X)\mid(f_1,\ldots,f_N)\in\overline{F_X}^{\,N}\}.
\]
\end{definition*}

\begin{definition*}[Unordered measurements of families of gd-sets]
For a family $\mathcal E$ of gd-sets, set
\[
\DM(\mathcal E;N,R)\coloneqq\bigcup_{X\in\mathcal E}\DM(X;N,R).
\]
\end{definition*}

\begin{lemma}
\label{lem:ordered-to-feature-measurement}
For any $\mu,\nu\in\M(N,R)$, we have
\[\Box(q_{N,R}(\mu),q_{N,R}(\nu))\leq2\dpr(\mu,\nu).\]
Moreover, for any gd-set $X$,
\begin{equation}
\label{eq:ordered-feature-closure}
\overline{q_{N,R}[\M(X;N,R)]}^{\Box}=\overline{\DM(X;N,R)}^{\Box}.
\end{equation}
For gd-sets $X,Y$,
\begin{equation}
\label{eq:ordered-controls-feature}
\haus{\Box}(\DM(X;N,R),\DM(Y;N,R))
\leq2\haus{\dpr}(\M(X;N,R),\M(Y;N,R)).
\end{equation}
Furthermore, for nonempty families $\mathcal E,\mathcal E'$ of gd-sets,
\begin{align}
\overline{q_{N,R}[\M(\mathcal E;N,R)]}^{\Box}
&=\overline{\DM(\mathcal E;N,R)}^{\Box},\label{eq:family-measurement-closure}\\
\haus{\Box}(\DM(\mathcal E;N,R),\DM(\mathcal E';N,R))
&\leq2\haus{\dpr}(\M(\mathcal E;N,R),\M(\mathcal E';N,R)).\label{eq:family-measurement-comparison}
\end{align}
\end{lemma}

\begin{lemma}
\label{lem:general-ordered-measurement-stability}
For every positive integer $N$, every $R>0$, and all gd-sets $X,Y$, we have
\begin{align}
\haus{\dpr}(\M(X;N,R),\M(Y;N,R))&\leq N\dconc(X,Y),\label{eq:general-dconc-controls-ordered}\\
\haus{\dpr}(\M(X;N,R),\M(Y;N,R))&\leq\Box(X,Y).\label{eq:general-box-controls-ordered}
\end{align}
\end{lemma}

\begin{lemma}
\label{lem:l-pyramid-measurement-closed}
Let $\TB\subset\L\subset\lipone(\R)$ be a monoidal subfamily, and let $\mathcal P$ be a pyramid in $\L\circ\D$. For every positive integer $N$ and every $R>0$,
\begin{equation}
\label{eq:l-ordered-feature-preimage}
\M(\mathcal P;N,R)=q_{N,R}^{-1}(\DM(\mathcal P;N,R)),
\end{equation}
and this set is compact with respect to $\dpr$.
\end{lemma}

\begin{lemma}
\label{lem:pointwise-closure-sequential}
Let $\mathcal H$ be a family of $1$-Lipschitz functions on a separable metric space $(Z,d)$. Then every element of the pointwise closure $\overline{\mathcal H}$ is the pointwise limit of a sequence in $\mathcal H$.
\end{lemma}

\begin{proof}
Let $\mathcal H$ be pointwise bounded, and let $D=\{z_j\}_{j=1}^{\infty}$ be dense in $Z$. Its pointwise closure $\overline{\mathcal H}$ is again pointwise bounded and consists of $1$-Lipschitz functions. If $h,k\in\overline{\mathcal H}$ and $d(z,z_j)<\varepsilon$, then
\[|h(z)-k(z)|\leq2\varepsilon+|h(z_j)-k(z_j)|.\]
Thus, convergence on $D$ is equivalent to pointwise convergence on $Z$, and the pointwise topology on $\overline{\mathcal H}$ is induced by the metric
\[\rho(h,k)\coloneqq\sum_{j=1}^{\infty}2^{-j}\min\{1,|h(z_j)-k(z_j)|\}.\]
It follows that every element of $\overline{\mathcal H}$ is the pointwise limit of a sequence in $\mathcal H$. For a family $\mathcal H$ that is not assumed to be pointwise bounded and a prescribed $h\in\overline{\mathcal H}$, fix $z_0\in Z$ and set $\mathcal H_h\coloneqq\{g\in\mathcal H\mid |g(z_0)-h(z_0)|\leq1\}$. Every pointwise neighborhood of $h$ can be refined by also controlling the value at $z_0$, so $h\in\overline{\mathcal H_h}$, while
\[|g(z)|\leq |h(z_0)|+1+d(z,z_0)\qquad(g\in\mathcal H_h).\]
Therefore, $\mathcal H_h$ is pointwise bounded, and the same sequentiality conclusion holds for every family of $1$-Lipschitz functions. This completes the proof.
\end{proof}

If $X$ and $Y$ are gd-sets satisfying $Y\preceq X$, then
\[
\M(Y;N,R)\subset\overline{\M(X;N,R)}^{\dpr}.
\]
Indeed, pull back each function along a domination and apply \Cref{lem:pointwise-closure-sequential} to approximate it pointwise by a sequence of elements of $F_X$. Therefore,
\begin{equation}
\label{eq:associated-pyramid-measurement}
\M(\mathcal P_{\L\circ\D}(X);N,R)=\overline{\M(X;N,R)}^{\dpr}.
\end{equation}

For pyramids $\mathcal P,\mathcal Q$ in $\L\circ\D$, define $d_{\Pi}^{\M}(\mathcal P,\mathcal Q)$ by the series in \Cref{eq:measurement-pyramid-metric}, using their gd-set measurements.

\begin{theorem}
\label{thm:general-l-ordered-measurement-metric}
The function $d_{\Pi}^{\M}$ defined above is a metric on the space of pyramids. The three conditions in \Cref{thm:pyramid-ordered-measurement-criterion} remain equivalent when qm-space pyramids are replaced by pyramids in $\L\circ\D$ and their measurements are interpreted as above.

Moreover, for $X_n,X\in\L\circ\D$, the following are equivalent.
\begin{enumerate}[label=\textup{(\alph*)}]
\item For every positive integer $N$ and every real number $R>0$,
\[\haus{\dpr}(\M(X_n;N,R),\M(X;N,R))\longrightarrow0.\]
\item The associated pyramid $\mathcal P_{\L\circ\D}(X_n)$ converges weakly to $\mathcal P_{\L\circ\D}(X)$.
\item
\[d_{\Pi}^{\M}(\mathcal P_{\L\circ\D}(X_n),\mathcal P_{\L\circ\D}(X))\longrightarrow0.\]
\end{enumerate}
\end{theorem}

More generally, let $F\colon\mathcal B\rightleftarrows\mathcal A\colon G$ be a pyramidal adjunction, and suppose that a metric $d_{\Pi}^{\M}$ on the pyramids in $\mathcal A$ metrizes weak convergence. For pyramids in $\mathcal B$, define
\[
d_{\Pi}^{\M}(\mathcal P,\mathcal Q)
\coloneqq d_{\Pi}^{\M}(F_{\#}\mathcal P,F_{\#}\mathcal Q).
\]
This formula extends the notation in \Cref{eq:measurement-pyramid-metric} to the pullback along $F_{\#}$ of the metric on the pyramids in $\mathcal A$. By \Cref{thm:abstract-pyramid-transport}, the map $F_{\#}$ is injective and preserves and reflects weak convergence. Thus, this is a metric on the pyramids in $\mathcal B$ and metrizes their weak convergence. For a finite composable family of pyramidal adjunctions, successive pullback agrees with pullback along the composite by \Cref{prop:pyramidal-adjunction-composition}.

We now extend the notation of \Cref{def:qm-pyramid-ordered-measurement}. Let $F\dashv G$ be either $\operatorname{Rep}\dashv\operatorname{Rec}$ or $\operatorname{Rep}^{+}\dashv\operatorname{Rec}^{+}$, let $\mathcal C$ be its domain, and for a pyramid $\mathcal P$ in $\mathcal C$ set
\[
\M(\mathcal P;N,R)\coloneqq\bigcup_{A\in\mathcal P}\M(F(A);N,R).
\]
If $F=\operatorname{Rep}^{+}$ and $\mathcal C=\Xp$, then $F_{\operatorname{Rep}^{+}(A)}=\lipplus(A)$, so $\M(F(A);N,R)=\M(A;N,R)$. Hence this general notation agrees with its qm-space specialization in \Cref{def:qm-pyramid-ordered-measurement}.

\begin{proposition}
\label{prop:space-pyramid-measurement-transport}
Let $F\dashv G$ and $\mathcal C$ be as above, and let $\mathcal P$ be a pyramid in $\mathcal C$. Then
\[\M(\mathcal P;N,R)=\M(F_{\#}\mathcal P;N,R),\]
and this set is compact with respect to $\dpr$. Moreover, for every object $A$ in the domain,
\[\M(\mathcal P_{\mathcal C}(A);N,R)=\M(F(A);N,R).\]
\end{proposition}

\begin{proof}
Set $\mathcal E_{\mathcal P}\coloneqq\{F(A)\mid A\in\mathcal P\}$, and write $\L_F\coloneqq\lipone(\R)$ for $F=\operatorname{Rep}$ and $\L_F\coloneqq\Lup$ for $F=\operatorname{Rep}^{+}$. Suppose that $A\in\mathcal P$ and $q_{N,R}(\nu)\in\DM(F(A);N,R)$. Then there are $f_1,\ldots,f_N\in\overline{F_{F(A)}}$ such that $\nu'\coloneqq(b_R\circ f_1,\ldots,b_R\circ f_N)_*\mu_{F(A)}$ satisfies $q_{N,R}(\nu')=q_{N,R}(\nu)$. The coordinate gd-sets $\mathsf B_N(\nu')$ and $\mathsf B_N(\nu)$ are therefore isomorphic. Choose the measure-preserving isometry $u\colon\supp\nu'\to\supp\nu$ in this isomorphism, so that $\overline{F_{\mathsf B_N(\nu)}}\circ u=\overline{F_{\mathsf B_N(\nu')}}$. The tuple $(\pr_1\circ u,\ldots,\pr_N\circ u)$ consists of features of the $\L_F$-saturation $Z\coloneqq\L_F\circ\mathsf B_N(\nu')$ and pushes $\nu'$ forward to $\nu$. The map $v\colon F(A)\to\mathsf B_N(\nu')$ given by the clipped tuple $(b_R\circ f_1,\ldots,b_R\circ f_N)$ is measure-preserving. Since $b_R\in\TB\subset\L_F$ and $F(A)$ is an $\L_F$-gd-set, pullback along $v$ sends every feature of $Z$ into $\overline{F_{F(A)}}$. Hence $Z\preceq F(A)$. By \Cref{eq:adjunction-order}, $B\coloneqq G(Z)\preceq A$, so $B\in\mathcal P$. The counit gives $Z\preceq F(B)$, and the corresponding domination map $F(B)\to Z$ pulls the tuple $(\pr_1\circ u,\ldots,\pr_N\circ u)$ back to a tuple in the feature family of $F(B)$. Here the feature family is pointwise closed because $F(B)$ is a representation. The pulled-back tuple pushes the measure of $F(B)$ forward to $\nu$, so $\nu\in\M(F(B);N,R)\subset\M(\mathcal P;N,R)$. Thus,
\[\M(\mathcal P;N,R)=q_{N,R}^{-1}(\DM(\mathcal E_{\mathcal P};N,R)).\]
The counit gives $Z\preceq FG(Z)$, and therefore
\[\DM(\mathcal E_{\mathcal P};N,R)=\DM(F_{\#}\mathcal P;N,R).\]
Combining these identities with \Cref{eq:l-ordered-feature-preimage} proves the first assertion.

The transported family $F_{\#}\mathcal P$ is a pyramid, so \Cref{lem:l-pyramid-measurement-closed} shows that the common measurement set is compact. Finally, if $B\preceq A$, then every $1$-Lipschitz or one-sided $1$-Lipschitz function on $B$, according to the category, pulls back along a domination to the corresponding function family of $F(A)$. This gives
\[\M(\mathcal P_{\mathcal C}(A);N,R)\subset\M(F(A);N,R).\]
The reverse inclusion follows from $A\in\mathcal P_{\mathcal C}(A)$.
This completes the proof.
\end{proof}

For an mm-space $X$, we have $F_{\operatorname{Rep}(X)}=\lipone(X)$. Thus, the ordered measurement used here agrees with the classical $(N,R)$-measurement \cite[Definition~5.37]{shioya2016mmg}.

\subsection{Measurement criteria for general \texorpdfstring{$\L$}{L}-pyramids}
\label{subsec:appendix-general-l-local-theory}

Throughout this subsection, $\TB\subset\L\subset\lipone(\R)$ is a monoidal subfamily. We compare ordered measurements with finite measurements that forget the order of the functions.

\begin{proposition}[{\cite[Proposition~7.2]{gds2}}]
\label{prop:gds2-finite-measurement-criterion}
Let $\mathcal P_n,\mathcal P$ be pyramids in $\L\circ\D$. The following are equivalent.
\begin{enumerate}[label=\textup{(\roman*)}]
\item $\mathcal P_n$ converges weakly to $\mathcal P$.
\item For every positive integer $N$ and every real number $R>0$,
\[\haus{\Box}(\DM(\mathcal P_n;N,R),\DM(\mathcal P;N,R))\longrightarrow0.\]
\item For every positive integer $N$, the convergence in \textup{(ii)} holds with $R=N$.
\end{enumerate}
\end{proposition}

\begin{proof}[Proof of \Cref{lem:l-pyramid-measurement-closed}]
The forward inclusion in \Cref{eq:l-ordered-feature-preimage} follows from the definitions. Conversely, suppose that $q_{N,R}(\nu)\in\DM(\mathcal P;N,R)$. For some $X\in\mathcal P$, there are $f_1,\ldots,f_N\in\overline{F_X}$ such that $\nu'\coloneqq(b_R\circ f_1,\ldots,b_R\circ f_N)_*\mu_X$ satisfies $q_{N,R}(\nu')=q_{N,R}(\nu)$. By the definition of gd-set isomorphism, there is a measure-preserving isometry $u\colon\supp\nu'\to\supp\nu$ such that $\overline{F_{\mathsf B_N(\nu)}}\circ u=\overline{F_{\mathsf B_N(\nu')}}$. The map $v\colon X\to\mathsf B_N(\nu')$ given by $v(x)=(b_R(f_1(x)),\ldots,b_R(f_N(x)))$ is measure-preserving. Since $b_R\in\TB\subset\L$, pullback along $v$ sends every feature of the $\L$-saturation $\L\circ\mathsf B_N(\nu')$ into $\overline{F_X}$, so $\L\circ\mathsf B_N(\nu')\preceq X$. The isomorphism $u$ identifies this saturation with $\L\circ\mathsf B_N(\nu)$; hence the latter also belongs to $\mathcal P$. Its coordinate functions realize $\nu$, which proves the reverse inclusion.

The set $\DM(\mathcal P;N,R)$ is compact and hence closed by \cite[Lemma~7.1]{gds2}. The map $q_{N,R}$ is continuous by \Cref{lem:ordered-to-feature-measurement}, so \Cref{eq:l-ordered-feature-preimage} shows that $\M(\mathcal P;N,R)$ is closed. It is compact because $\M(N,R)$ is compact \cite[Lemma~1.17(3) and Definition~5.37]{shioya2016mmg}. This completes the proof.
\end{proof}

For pyramids $\mathcal P,\mathcal Q$ in $\L\circ\D$, a positive integer $N$, and $R>0$, set
\[
\begin{aligned}
a_{N,R}&\coloneqq\haus{\dpr}(\M(\mathcal P;N,R),\M(\mathcal Q;N,R)),\\
b_{N,R}&\coloneqq\haus{\Box}(\DM(\mathcal P;N,R),\DM(\mathcal Q;N,R)).
\end{aligned}
\]

\begin{lemma}
\label{lem:l-ordered-feature-bilipschitz}
Let $\mathcal P,\mathcal Q$ be pyramids in $\L\circ\D$, let $N$ be a positive integer, and let $R>0$. Then
\[\frac12b_{N,R}\leq a_{N,R}\leq b_{N,R}.\]
\end{lemma}

\begin{proof}
By \Cref{lem:l-pyramid-measurement-closed}, the ordered measurement sets of $\mathcal P$ and $\mathcal Q$ are compact. The first inequality follows from \Cref{eq:family-measurement-comparison}. To prove the second, take $\nu\in\M(\mathcal P;N,R)$. Choose $X\in\mathcal P$ and $f_1,\ldots,f_N\in F_X$ such that
\[\nu=(b_R\circ f_1,\ldots,b_R\circ f_N)_*\mu_X.\]
Then $q_{N,R}(\nu)\in\DM(\mathcal P;N,R)$. Since $\DM(\mathcal Q;N,R)$ is compact \cite[Lemma~7.1]{gds2}, for every $\varepsilon>0$ there exists $B\in\DM(\mathcal Q;N,R)$ such that
\[\Box(q_{N,R}(\nu),B)\leq b_{N,R}+\varepsilon.\]
The coordinate functions of $\mathsf B_N(\nu)$ belong to its function family, and repeated coordinates may be chosen repeatedly in an ordered measurement. Therefore, $\nu\in\M(q_{N,R}(\nu);N,R)$. By \Cref{lem:general-ordered-measurement-stability},
\[
\dpr\bigl(\nu,\M(B;N,R)\bigr)
\leq\Box(q_{N,R}(\nu),B)\leq b_{N,R}+\varepsilon.
\]
Since $B\in\DM(\mathcal Q;N,R)$, there exists $Y\in\mathcal Q$ such that $B\in\DM(Y;N,R)$. The saturation $\L\circ B$ is an $\L$-gd-set dominated by $Y$ because $b_R\in\TB\subset\L$ and $Y$ is an $\L$-gd-set. Hence $\L\circ B\in\mathcal Q$. Since $\operatorname{id}_{\R}\in\L$,
\[
\M(B;N,R)\subset\M(\L\circ B;N,R)
\subset\M(\mathcal Q;N,R).
\]
It follows that
\[
\dpr\bigl(\nu,\M(\mathcal Q;N,R)\bigr)
\leq b_{N,R}+\varepsilon.
\]
Letting $\varepsilon\downarrow0$, taking the supremum over $\nu$, and interchanging $\mathcal P$ and $\mathcal Q$ gives $a_{N,R}\leq b_{N,R}$. This completes the proof.
\end{proof}

\begin{proof}[Proof of \Cref{thm:general-l-ordered-measurement-metric}]
For each fixed $N$ and $R$, \Cref{lem:l-ordered-feature-bilipschitz} shows that the ordered and unordered Hausdorff distances converge to zero simultaneously. Thus, \Cref{prop:gds2-finite-measurement-criterion} shows that \textup{(i)} and \textup{(ii)} are equivalent, and that they are also equivalent to convergence of all diagonal measurements.

The space $\M(N,N)$ is compact \cite[Lemma~1.17(3) and Definition~5.37]{shioya2016mmg} and $\dpr\leq1$, so the series converges. Condition \textup{(iii)} implies convergence of every diagonal measurement. Conversely, if all diagonal measurements converge, split the series into the terms with $N\leq K$ and those with $N>K$. The first part is a finite sum and tends to zero, while the second part is bounded by
\[\sum_{N>K}\frac{1}{N\cdot2^N}.\]
This proves \textup{(iii)} by letting $K\to\infty$. To prove separation, apply $d_{\Pi}^{\M}(\mathcal P,\mathcal Q)=0$ to a constant sequence and use the inner condition for weak convergence in both directions.

Finally, the Hausdorff distance is unchanged when sets are replaced by their closures. By \Cref{eq:associated-pyramid-measurement}, condition \textup{(a)} is equivalent to convergence of all ordered measurements of the associated pyramids. The equivalence already proved for pyramids shows that this is equivalent to \textup{(b)} and \textup{(c)}.
This completes the proof.
\end{proof}

\subsection{Proofs of the ordered-measurement estimates}
\label{subsec:appendix-ordered-measurement-proofs}

\begin{proof}[Proof of \Cref{lem:ordered-to-feature-measurement}]
We first prove the following estimate. Let $F$ be a nonempty family of real-valued functions on $Z$ such that
\[
d_F^{\mathrm{feat}}(x,y)\coloneqq\sup_{f\in F}|f(x)-f(y)|
\]
is a complete separable metric on $Z$. For Borel probability measures $\mu,\nu$ on $(Z,d_F^{\mathrm{feat}})$, set
\[
Z_{F,\mu}\coloneqq(\supp\mu,F|_{\supp\mu},\mu),\qquad Z_{F,\nu}\coloneqq(\supp\nu,F|_{\supp\nu},\nu).
\]
Write $d_{\operatorname P}^{d_F^{\mathrm{feat}}}$ for the Prokhorov distance with respect to $d_F^{\mathrm{feat}}$. Then these are gd-sets, and
\[\Box(Z_{F,\mu},Z_{F,\nu})\leq2d_{\operatorname P}^{d_F^{\mathrm{feat}}}(\mu,\nu).\]

The sets $\supp\mu$ and $\supp\nu$ are closed subsets of $(Z,d_F^{\mathrm{feat}})$, so $Z_{F,\mu}$ and $Z_{F,\nu}$ are gd-sets. Every member of $F$ is $1$-Lipschitz with respect to $d_F^{\mathrm{feat}}$. By separability and a diagonal subsequence argument, every nonempty subset $A\subset Z$ satisfies
\[\overline{F|_A}=\{f|_A\mid f\in\overline F\}.\]
Indeed, \Cref{lem:pointwise-closure-sequential}, applied on $Z$, shows that every $f\in\overline F$ is the pointwise limit of a sequence in $F$. This proves the inclusion from right to left. Conversely, $A$ is separable, so \Cref{lem:pointwise-closure-sequential}, applied on $A$, shows that every $h\in\overline{F|_A}$ is the pointwise limit of a sequence $f_n|_A$ with $f_n\in F$. The values of $f_n$ are bounded at one point of $A$, and the $1$-Lipschitz property yields a subsequence that converges on a countable dense subset of $Z$. The dense-set estimate above shows that this subsequence converges pointwise on $Z$ to the unique $1$-Lipschitz extension $f$ of its limit. Then $f\in\overline F$ and $f|_A=h$, which proves the reverse inclusion.

Let $\varepsilon>d_{\operatorname P}^{d_F^{\mathrm{feat}}}(\mu,\nu)$. By Strassen's theorem \cite{strassen1965existence}, there exists a coupling $\pi$ of $\mu$ and $\nu$ such that
\[
\pi(S)\geq1-\varepsilon,\qquad S\coloneqq\{(x,y)\in\supp\mu\times\supp\nu\mid d_F^{\mathrm{feat}}(x,y)\leq\varepsilon\}.
\]
The set $S$ is closed. For every $f\in\overline F$, we have $\dinf{S}(f\circ\pr_1,f\circ\pr_2)\leq\varepsilon$. Therefore,
\[
\haus{\dinf{S}}(\overline{F|_{\supp\mu}}\circ\pr_1,\overline{F|_{\supp\nu}}\circ\pr_2)\leq\varepsilon.
\]
It follows from \Cref{eq:gds-box} that $\Box(Z_{F,\mu},Z_{F,\nu})\leq2\varepsilon$. Letting $\varepsilon\downarrow d_{\operatorname P}^{d_F^{\mathrm{feat}}}(\mu,\nu)$ proves the estimate.

Apply this estimate with $Z=[-R,R]^N$ and $F=\{\pr_1,\ldots,\pr_N\}$. The induced metric is the $\ell^\infty$ distance, and the first inequality follows.

Since $F_X\subset\overline{F_X}$, the left-hand side of \Cref{eq:ordered-feature-closure} is contained in the right-hand side. For the reverse inclusion, take $f_1,\ldots,f_N\in\overline{F_X}$. By \Cref{lem:pointwise-closure-sequential}, for each $i$ choose a sequence $f_{i,k}\in F_X$ converging pointwise to $f_i$. The integrands are bounded by clipping, so the dominated convergence theorem applied to any bounded continuous function on $[-R,R]^N$ gives the weak convergence
\[
(b_R\circ f_{1,k},\ldots,b_R\circ f_{N,k})_*\mu_X
\longrightarrow
(b_R\circ f_1,\ldots,b_R\circ f_N)_*\mu_X.
\]
Since $\dpr$ metrizes weak convergence on $\M(N,R)$ \cite[Lemma~1.17(1)]{shioya2016mmg}, the first inequality shows that their images converge with respect to $\Box$. This proves \Cref{eq:ordered-feature-closure}. The Hausdorff distance is unchanged when sets are replaced by their closures. Applying the first inequality in both directions therefore gives \Cref{eq:ordered-controls-feature}.

Taking the union of \Cref{eq:ordered-feature-closure} over $Y\in\mathcal E$ and then taking the closure gives \Cref{eq:family-measurement-closure}. The first estimate also gives
\[\haus{\Box}(q_{N,R}[A],q_{N,R}[B])\leq2\haus{\dpr}(A,B)\]
for any two subsets $A,B\subset\M(N,R)$. Set $A=\M(\mathcal E;N,R)$ and $B=\M(\mathcal E';N,R)$. The invariance of the Hausdorff distance under taking closures, together with \Cref{eq:family-measurement-closure}, gives \Cref{eq:family-measurement-comparison}. This completes the proof.
\end{proof}

\begin{definition*}[Subtransport plans {\cite[Definitions~1.20--1.21 and Theorem~1.22]{shioya2016mmg}}]
For Borel probability measures $\alpha$ and $\beta$ on a metric space $(Z,d_Z)$, a Borel measure $m$ on $Z\times Z$ is a \emph{subtransport plan} between them if its first and second marginals are Borel measures $\alpha'\leq\alpha$ and $\beta'\leq\beta$, respectively.
\end{definition*}

\begin{definition*}[$\varepsilon$-subtransport plans {\cite[Theorem~1.22]{shioya2016mmg}}]
A subtransport plan is an \emph{$\varepsilon$-subtransport plan} if its support is contained in $\{(z,z')\in Z\times Z\mid d_Z(z,z')\leq\varepsilon\}$.
\end{definition*}

\begin{definition*}[Deficiency {\cite[Theorem~1.22]{shioya2016mmg}}]
The \emph{deficiency} of a subtransport plan $m$ is $1-m(Z\times Z)$.
\end{definition*}

If $Z$ is complete and separable, the subtransport-plan form of Strassen's theorem \cite{strassen1965existence}\cite[Theorem~1.22]{shioya2016mmg} states that $\dpr(\alpha,\beta)$ is the infimum of the positive $\varepsilon$ for which there is an $\varepsilon$-subtransport plan of deficiency at most $\varepsilon$.

\begin{proof}[Proof of \Cref{lem:general-ordered-measurement-stability}]
Take $\eta>\dconc(X,Y)$. By \Cref{eq:gds-dconc}, there exists $\pi\in\Tplan(\mu_X,\mu_Y)$ such that
\[\haus{\kf^\pi}(F_X\circ\pr_1,F_Y\circ\pr_2)<\eta.\]
Fix $f_1,\ldots,f_N\in F_X$, and for each $i$ choose $g_i\in F_Y$ such that
\[\kf^\pi(f_i\circ\pr_1,g_i\circ\pr_2)<\eta.\]
Since the clipping map $b_R$ is $1$-Lipschitz, the same inequality holds for $b_R\circ f_i$ and $b_R\circ g_i$. Therefore,
\[\pi\left(\max_{1\leq i\leq N}|b_R\circ f_i\circ\pr_1-b_R\circ g_i\circ\pr_2|>\eta\right)\leq N\eta.\]
Push this coupling forward by the two coordinate maps. Restricting the resulting coupling to the set where the $\ell^\infty$ distance is at most $\eta$ gives an $N\eta$-subtransport plan with deficiency at most $N\eta$, so the subtransport-plan form of Strassen's theorem \cite{strassen1965existence}\cite[Theorem~1.22]{shioya2016mmg} gives
\[\dpr((b_R\circ f_1,\ldots,b_R\circ f_N)_*\mu_X,(b_R\circ g_1,\ldots,b_R\circ g_N)_*\mu_Y)\leq N\eta.\]
Interchanging $X$ and $Y$ and letting $\eta\downarrow\dconc(X,Y)$ proves \Cref{eq:general-dconc-controls-ordered}.

Next, let $\eta>\Box(X,Y)$. By \Cref{eq:gds-box}, choose $\pi\in\Tplan(\mu_X,\mu_Y)$ and a closed set $S\subset X\times Y$ such that
\[
1-\pi(S)<\eta,\qquad
2\haus{\dinf{S}}(\overline{F_X}\circ\pr_1,\overline{F_Y}\circ\pr_2)<\eta.
\]
For arbitrary $f_1,\ldots,f_N\in F_X$, there exist $g_1,\ldots,g_N\in\overline{F_Y}$ satisfying
\[
\dinf{S}(f_i\circ\pr_1,g_i\circ\pr_2)<\frac{\eta}{2}
\qquad(i=1,\ldots,N).
\]
Thus, the distance between the two clipped coordinate maps is less than $\eta/2$ on $S$, and restricting their pushed-forward coupling to the set where this distance is at most $\eta/2$ gives an $\eta$-subtransport plan with deficiency less than $\eta$, so the subtransport-plan form of Strassen's theorem \cite{strassen1965existence}\cite[Theorem~1.22]{shioya2016mmg} shows that the Prokhorov distance between the pushforward measures is at most $\eta$. By \Cref{lem:pointwise-closure-sequential}, choose $g_{i,k}\in F_Y$ converging pointwise to $g_i$ for every $i$. The corresponding pushforward measures converge with respect to $\dpr$ by \cite[Lemma~1.17(1)]{shioya2016mmg}. Since the distance to a set equals the distance to its closure, the distance from the first measure to $\M(Y;N,R)$ is at most $\eta$. The reverse direction is identical, and taking the two suprema gives \Cref{eq:general-box-controls-ordered}. This completes the proof.
\end{proof}

\subsection{Specialization to qm-spaces}
\label{subsec:appendix-qm-measurement-specialization}

\begin{proof}[Proof of \Cref{lem:ordered-measurement-stability}]
Apply \Cref{eq:general-dconc-controls-ordered} to $\operatorname{Rep}^{+}(X)$ and $\operatorname{Rep}^{+}(Y)$ and use \Cref{def:pullback-observable-distance}. This gives the required estimate. This completes the proof.
\end{proof}

\begin{proof}[Proof of \Cref{thm:pyramid-ordered-measurement-criterion}]
By \Cref{thm:qm-pyramid-transport}, the convergence $\mathcal P_n\to\mathcal P$ is equivalent to
\[
(\operatorname{Rep}^{+})_{\#}\mathcal P_n
\longrightarrow
(\operatorname{Rep}^{+})_{\#}\mathcal P.
\]
The qm-space specialization of \Cref{prop:space-pyramid-measurement-transport} gives
\[
\M(\mathcal P;N,R)
=\M((\operatorname{Rep}^{+})_{\#}\mathcal P;N,R).
\]
The equivalence of the three conditions and the metric property now follow from \Cref{thm:general-l-ordered-measurement-metric} with $\L=\Lup$. This completes the proof.
\end{proof}

\section*{Acknowledgments}

The author would like to thank Professor Takashi Shioya for many helpful suggestions and guidance. The author used Claude and GPT-5.6-series Codex models as AI-assisted tools in preparing this manuscript. The author reviewed and revised the mathematical content and takes full responsibility for the final manuscript.

\end{document}